\documentclass[12pt]{amsart}
\usepackage{tikz-cd}
\usepackage{amssymb}
\usepackage{amsmath,amscd}
\usepackage{mathrsfs}
\usepackage{mathtools}
\usepackage{url}
\usepackage{cite}
\usepackage{fullpage}
\usepackage[hidelinks]{hyperref}
\usepackage{verbatim}
\usepackage{comment}
\usepackage{fancyvrb}
\usepackage{fvextra}
\usepackage{caption}
\usepackage{tikz}
\usepackage{tkz-graph}
\usetikzlibrary{matrix}
 \usetikzlibrary{shapes}
\usetikzlibrary{arrows,decorations.markings}
\usepackage{tikz-cd}
\usepackage{stmaryrd}
\usepackage{pdfpages}
\usepackage{color, colortbl}
\usepackage{array, adjustbox}
\usepackage{tabularray}

\pgfarrowsdeclare{bad to}{bad to}
{
  \pgfarrowsleftextend{-2\pgflinewidth}
  \pgfarrowsrightextend{\pgflinewidth}
}
{
  \pgfsetlinewidth{0.8\pgflinewidth}
  \pgfsetdash{}{0pt}
  \pgfsetroundcap
  \pgfsetroundjoin
  \pgfpathmoveto{\pgfpoint{-3\pgflinewidth}{4\pgflinewidth}}
  \pgfpathcurveto
  {\pgfpoint{-2.75\pgflinewidth}{2.5\pgflinewidth}}
  {\pgfpoint{0pt}{0.25\pgflinewidth}}
  {\pgfpoint{0.75\pgflinewidth}{0pt}}
  \pgfpathcurveto
  {\pgfpoint{0pt}{-0.25\pgflinewidth}}
  {\pgfpoint{-2.75\pgflinewidth}{-2.5\pgflinewidth}}
  {\pgfpoint{-3\pgflinewidth}{-4\pgflinewidth}}
  \pgfusepathqstroke
}

\newtheorem{thm}{Theorem}[section]
\newtheorem{prop}[thm]{Proposition}

\theoremstyle{definition}

\numberwithin{equation}{section}

\newcommand{\VV}{\mathcal{V}}

\newcommand{\NN}{\mathbb{N}}

\renewcommand{\ss}{\mathbb{S}}

\newcommand{\M}{\mathcal{M}}

\newcommand{\A}{\mathcal{A}}

\renewcommand{\tilde}{\widetilde}

\newcommand{\lstw}{\mathsf{L}\mathsf{S}_{12}}

\usepackage{wrapfig}

\renewcommand{\bar}{\overline}
\newcommand{\Mb}{\overline{\M}}

\usepackage{ytableau,tikz,varwidth, tikz-cd}
\usepackage{hyperref}

\usetikzlibrary{calc, decorations, positioning}
\usepackage[enableskew]{youngtab}
\usepackage{enumerate}
\usepackage{amssymb, mathscinet}

\usepackage[all,cmtip]{xy}

\usepackage{multicol}
\usepackage{multirow, bigstrut}
\usepackage{enumitem}

\usetikzlibrary{matrix}
\usetikzlibrary{shapes}
\usetikzlibrary{arrows, automata}
\usetikzlibrary{decorations,decorations.pathmorphing,decorations.pathreplacing,decorations.markings}
\usetikzlibrary{through}
\usetikzlibrary{decorations.pathreplacing,calligraphy}

\tikzset{every picture/.style={baseline=-.65ex} }
\tikzset{ext/.style={circle, draw,inner sep=1pt,scale=1.4},int/.style={circle,draw,fill,inner sep=1pt},nil/.style={inner sep=1pt}}

\tikzset{every loop/.style={draw}}
\tikzset{
  crossed/.style={
    decoration={markings,mark=at position .5 with {\arrow{|}}},
    postaction={decorate},
    shorten >=0.4pt}}

\tikzset{->-/.style={decoration={
    markings,
    mark=at position .9 with {\arrow{>}}},postaction={decorate}}}

\tikzstyle{every node}=[scale=1.2]

\newcommand{\CC}{\mathbb{C}}

\newcommand{\Q}{\mathbb{Q}}

\newcommand{\cM}{\mathcal{M}}

\newcommand{\sgn}{\operatorname{sgn}}

\newcommand{\MM}{\overline{\mathcal{M}}}

\newcommand{\bGK}{\overline{\GK}{}}
\newcommand{\hide}[1]{}

\newtheorem{Fact}[thm]{Fact}

\newtheorem{Theorem}[thm]{Theorem}
\newenvironment{theorem}
  {\begin{Theorem}}{\end{Theorem}}
  
\newtheorem{Lemma}[thm]{Lemma}
\newenvironment{lemma}
  {\begin{Lemma}}{\end{Lemma}}

\newtheorem{Remark}[thm]{Remark}
\newenvironment{remark}
  {\begin{Remark}\rm}{\end{Remark}}

\newtheorem{Proposition}[thm]{Proposition}
\newenvironment{proposition}
  {\begin{Proposition}}{\end{Proposition}}

\newtheorem{Corollary}[thm]{Corollary}
\newenvironment{corollary}
  {\begin{Corollary}}{\end{Corollary}}

\newtheorem{Conjecture}[thm]{Conjecture}

  \theoremstyle{definition} 

\theoremstyle{remark}

\newcommand{\gr}{\mathrm{gr}}

\newcommand{\GK}{\mathsf{GK}}

\newcommand{\myB}{\bGK^{12,1}}

\title{Getzler-Kapranov Graph Complex cohomology computations in weight 13}

\author{Marco Belli}

\thanks{
  This paper has been written as part of a semester project for the Master's of Science of Marco Belli at ETH Zürich. Professor Thomas Willwacher acted as a supervisor of this project.
}

\begin{document}

\begin{abstract}
    Following the approach of \cite{CLPW2} for computing weight graded pieces of $H^*_c(\M_{g,n})$ using the Getzler-Kapranov graph complex, we extend their results in weight 13 for the $(g,n)$ pairs with $3g+2n=28$.
\end{abstract}

\maketitle


\vspace{-.3in}

\section{Introduction}

The authors of \cite{CLPW2} study the weight graded pieces of $H^*_c(\M_{g,n})$ using graph complexes. The associated graded of the weight filtration is identified with the cohomology of the modular cooperad whose $(g,n)$ part is $H^*(\Mb_{g,n})$, it is known as the Getzler-Kapranov graph complex:
\begin{equation}
    H^*(\GK_{g,n}^k) \cong  \gr^W_k H^*_c(\M_{g,n}) := W_k H^*_c(\M_{g,n})/W_{k-1}H^*_c(\M_{g,n}).
\end{equation}

They obtain the two following results in weight 13.
\begin{prop}\label{prop:wt 13 vanishing}
    If $3g + 2n \leq 25$ then $\gr_{13}^W H^*_c(\M_{g,n}) = 0$.
\end{prop}

\begin{thm} \label{thm:lowexc13}
    Suppose $3g + 2n \in \{26, 27\}$. Then $\gr_{13}^W H^*_c(\M_{g,n})$ is nonzero only in degree $$k(g,n) = 3g + n - 2 - \delta_{0,n},$$ and there is an $\ss_n$-equivariant isomorphism $\gr_{13}^W H^{k(g,n)}_c(\cM_{g,n}) \cong Z_{g,n} \otimes \lstw$, where
\begin{align*}
    Z_{1,12} & \cong V_{21^{10}} & Z_{2,10} & \cong V_{1^{10}} & Z_{3,9} & \cong V_{1^{9}} \\
    Z_{4,7} & \cong V_{1^{7}} & Z_{5,6} & \cong V_{1^{6}} \oplus V_{21^4}^{\oplus 2} & Z_{6,4} & \cong V_{1^{4}} \\ Z_{7,3} & \cong V_{1^{3}} \oplus V_{21}^{\oplus 2}& Z_{8,1} & \cong \Q & Z_{9,0} & \cong \Q 
\end{align*}
\end{thm}
Recall that the action of $\ss_n$ on all cohomology groups is induced by the permutation of the marked points in $\M_{g,n}$. The resulting vector spaces are decomposed according to the irreducible representations $V_{\lambda}$ of $\ss_n$ corresponding to each partition $\lambda$ of $n$.

With our computations we extend to the following case.
\begin{thm} \label{thm:excess28}
    Suppose $3g + 2n =28$. Then $\gr_{13}^W H^*_c(\M_{g,n})$ vanishes outside the degrees $$k_1(g,n) = 3g + n - 2 \hspace{.5cm}\text{and}\hspace{.5cm} k_2(g,n) = 3g + n - 3$$ and there are $\ss_n$-equivariant isomorphisms $$ \gr_{13}^W H^{k_1(g,n)}_c(\cM_{g,n}) \cong Z_{g,n} \otimes \lstw  \hspace{1cm} \gr_{13}^W H^{k_2(g,n)}_c(\cM_{g,n}) \cong W_{g,n} \otimes \lstw , $$ where
\begin{align*}
    Z_{2,11} &\cong V_{21^{9}}\oplus V_{221^7} \oplus V_{31^8}^{\oplus 2} & Z_{4,8} &\cong V_{21^6}\oplus V_{221^4}\oplus V_{31^5}^{\oplus 3} & Z_{6,5} &\cong V_{21^3}\oplus V_{221}\oplus V_{31^2}^{\oplus 3}
\end{align*}
\begin{align*}
    Z_{8,2}&=0 &  W_{2,11}&=V_{1^{11}}^{\oplus 2} & W_{4,8}&=V_{1^8}^{\oplus 2} & W_{6,5}&=V_{1^5}^{\oplus 2} & W_{8,2}&=V_{1^2}
\end{align*}
\end{thm}

\section{Graph Complexes in weight 13}

In this section, we recall the definition of the Getzler-Kapranov graph complex and its simplified version obtained as a quasi-isomorphic quotient in \cite[Section 2]{CLPW2}, as well as the presentations of the relevant cohomology groups.

\subsection{The Getzler-Kapranov graph complex in weight 13}
In the following we will consider stable graphs $\Gamma\in\Gamma((g,n))$ equipped with a decoration $(\gamma_v)_{v\in V(\Gamma)}$ of their vertices by cohomology classes $\gamma_v \in H^{*}(\MM_{g_v,n_v})$ which behaves tensorially in $v$, and with an ordering of the internal edges $(e_i)_{i\leq |E(\Gamma)|}$ subject to the relations $(\sigma e_i)_i=sgn(\sigma)(e_i)_i$ for all $\sigma\in\ss_{E(\Gamma)}$; we will refer to such graphs $(\Gamma,\gamma\otimes e)$ shortly as \emph{decorated graphs}. Isomorphisms $\phi:\Gamma\xrightarrow{\sim}\Gamma'$ of stable graphs act on the decorated graphs with underlying graph $\Gamma$ by traslations:
\[ \phi_* (\Gamma,(\gamma_v)_{v}\otimes (e_i)_i) = (\Gamma',(\gamma_v)_{\phi(v)}\otimes (\phi e_i)_i). \]

The Getzler-Kapranov graph complex has a presentation by the space of coinvariants of decorated graphs under isomorphisms:
\begin{multline}\label{equ:GKdef}
    \GK_{g,n} = 
    \bigg( \bigoplus_{\Gamma\in\Gamma((g,n))} \bigotimes_{v\in V(\Gamma)}H^{*}(\MM_{g_v,n_v})\otimes \Q[-1]^{\otimes |E(\Gamma)|} \bigg) \Big/ \sim \\
    = \bigoplus_{[\Gamma]\in[\Gamma((g,n))]} \bigg(
    \bigotimes_{v\in V(\Gamma)}H^{*}(\MM_{g_v,n_v})\otimes \Q[-1]^{\otimes |E(\Gamma)|} \bigg) \Big/ \text{Aut}_{\Gamma},
\end{multline}
where the $\Gamma((g,n))$ are stable graphs and $[\Gamma((g,n))]$ their isomorphism classes, of which there are only finitely for every pair $(g,n)$. Automorphisms of stable graphs are understood to fix the $n$ labeled hairs and to act by the sign of their induced permutation on the internal edges $E(\Gamma)$.

The \emph{weight} of a decorated graph $(\Gamma,\gamma\otimes e)$ with all $\gamma_v$ of pure cohomological degree $k_v$ is $\kappa=\sum_v k_v$. For each $\kappa\in\NN$, the homogeneous slice of weight $\kappa$ of $\GK_{g,n}$ is then given by 
\begin{equation}\label{equ:GKweight}
    \GK^{\kappa}_{g,n} 
    = \bigoplus_{[\Gamma]\in[\Gamma((g,n))]} \bigoplus_{\substack{\lambda\,\text{partition} \\ \text{of}\,\kappa}} \bigg( \bigoplus_{\substack{k:V(\Gamma)\rightarrow \NN \\ k\,\text{respects}\,\lambda}}
    \bigotimes_{v\in V(\Gamma)}H^{k_v}(\MM_{g_v,n_v})\otimes \Q[-1]^{\otimes |E(\Gamma)|} \bigg) \Big/ \text{Aut}_{\Gamma},
\end{equation}

The grading of the complex $\GK_{g,n}$ is given by the weight plus the number of internal edges of the stable graphs. For the definition of the differential we refer to \cite{CLPW2}; in \ref{subsec:differential} we will describe its action on the graphical depictions of the generators.

Since $H^{k}(\MM_{g,n})=0$ for odd $k<11$ \cite{BergstromFaberPayne}, the only weight labelings $k:V(\Gamma)\rightarrow\NN$ with $13=\sum_v k_v$ giving rise to non trivial decorated graphs are the ones with zero weight on every vertex except either $(A)$ $k_{\bar{v}}=13$ or $(B)$ $k_{\bar{v}}=11$, $k_{\tilde{v}}=2$ for exactly one special vertex $\bar{v}$ and in case $(B)$ a second vertex $\tilde{v}$. We will denote by $N_{\bar{v}}$ and $N_{\tilde{v}}$ the hairs and half-edges at $\bar{v}$ and $\tilde{v}$. We can reduce further the coinvariants by summing over all possible choices of $\bar{v}$ in case $(A)$ and $\bar{v},\tilde{v}$ in case $(B)$, and requiring that automorphisms fix these vertices:

\begin{multline}\label{equ:GK13}
    \GK^{13}_{g,n}
    = \bigoplus_{\substack{ [\Gamma]\in[\Gamma((g,n))] \\ \bar{v}\in V(\Gamma)}}
    \left( H^{13}(\MM_{g_{\bar{v}},n_{\bar{v}}}) \otimes \Q[-1]^{\otimes |E(\Gamma)|} \right)/\text{Aut}_{\Gamma,\bar{v}} \\
    \oplus \bigoplus_{\substack{ [\Gamma]\in[\Gamma((g,n))] \\ \bar{v}\neq\tilde{v}\in V(\Gamma)}}
    \left( H^{11}(\MM_{g_{\bar{v}},n_{\bar{v}}}) \otimes H^{2}(\MM_{g_{\tilde{v}},n_{\tilde{v}}}) \otimes \Q[-1]^{\otimes |E(\Gamma)|} \right)/\text{Aut}_{\Gamma,\bar{v},\tilde{v}} .
\end{multline}

\subsection{Presentations of cohomology groups} \label{subsec:CohomologyGroups}

The $\Q$-Hodge structure on the cohomology of the moduli space of curves delivers the following decompositions.
\begin{equation} \label{equ:GroupsHodgeDeco}
    H^{13}(\MM_{g,n})\otimes\CC = H^{12,1}(\MM_{g,n})\oplus H^{1,12}(\MM_{g,n})
\end{equation}
\[ H^{11}(\MM_{g,n})\otimes\CC = H^{11,0}(\MM_{g,n})\oplus H^{0,11}(\MM_{g,n}) \]
\[ H^{2}(\MM_{g,n})\otimes\CC = H^{1,1}(\MM_{g,n}) \]
In this paper we are interested in cohomology classes $\gamma_v\in H^k(\MM_{g,n})$ when viewed as decorations of a vertex of genus $g$ and valence $n$ in some graph $\Gamma$. Since it is cumbersome to bring along $\gamma$ whenever we want to reference a specific decorated graph $(\Gamma,\gamma)$, we will encode this datum in symbols drawn onto the vertex being decorated. The goal is to have graphical depictions that determine uniquely all relevant cohomology classes, so that just by the drawing it is possible to unambiguosly determine what is the decoration at each vertex.

We will describe these graphical depictions for the one-vertex decorated graphs $(*_{g,n},\gamma)$ of genus $g$ and with $n$ hairs, labeled by a set $N$. The depiction of a general decorated graph $(\Gamma,\gamma\otimes e)$ is understood to be given by the one of each one-vertex graph $(*_{g_v,n_v},\gamma_v)$ for every vertex $v$ of $\Gamma$.

To understand automorphisms of decorated graphs we will have to keep track of the $\ss_n$-action on each cohomolgy group induced by a permutation of $N$, which in some cases involves a sign representation.

In the following drawings, solid lines represent the minimum amount of edges \emph{necessary} for the considered class to exist, whereas dashed lines represent a \emph{potential} existence of edges.

\subsubsection{The case $k=0$, $g=0$, $n\geq 3$}
We have $H^{0,0}(\MM_{0,n})=\CC$ and thus we can draw weight $0$ vertices without any graphical depiction.
\[
\begin{tikzpicture}
    \node[int] (v) at (0,0) {};
    \draw (v) edge[dashed] +(.4,0)  edge +(-.4,.4) edge[dashed] +(.4,.4);
    \draw (v) edge +(-.4,0) edge +(-.4,-.4)  edge[dashed] +(.4,-.4);
\end{tikzpicture}
:= (*_{0,n},1)
\]

\subsubsection{The case $k=(1,1)$, $g=1$,$n=1$}
$H^{1,1}(\MM_{1,1})$ is one dimensional, spanned by a class that we call $\delta_{irr}$. Since this is the only case where a non special vertex might have genus $1$, we will introduce a symbolic loop with a crossed edge and draw the node black as if it were a genus $0$ vertex. As $n=1$, there is no $\ss_n$-action on $\delta_{irr}$ to talk about.
\[
\begin{tikzpicture}
    \node[int] (v) at (0,0) {};
    \draw (v) edge +(0,-.3) edge[loop, crossed, distance=0.8cm] (v);
\end{tikzpicture}
:= (*_{1,1},\delta_{irr})
\]

\subsubsection{The case $k=(1,1)$, $g=0$} For this case we refer to \cite[Section 3]{PayneWillwacher21}.
In $g=0$, $H^{1,1}(\MM_{0,n})$ is non zero only for $n\geq 4$; so we operate under this assumption.
The group is generated by classes $\psi_i$ for every $1\leq i\leq n$ and $\delta\{\substack{A \\ A'}\}$ for every partition $A\sqcup A'=N$ with $|A|,|A'|\geq 2$. To depict $\delta\{\substack{A \\ A'}\}$ we split symbolically the vertex in two parts connected by a crossed edge and draw on one side the subset of hairs $A$ and on the other $A'$. The notation $\{\substack{A \\ A'}\}$ is chosen to signify that swapping $A$ and $A'$ doesn't change the class, which graphically means it doesn't matter on which sides the two sets of hairs are chosen to be drawn. The $\ss_n$-action is given by $\sigma\,\psi_i = \psi_{\sigma i}$ and $\sigma\,\delta\{\substack{A \\ A'}\}=\delta\{\substack{\sigma A \\ \sigma A'}\}$.
\[
    \begin{tikzpicture}
        \node[int] (v) at (0,0) {};
        \draw (v) edge[dashed] +(.6,0)  edge +(-.6,.5) edge[->-] +(.6,.5);
        \draw (v) edge[dashed] +(-.6,0) edge +(-.6,-.5)  edge +(.6,-.5);
        \node (x) at (.8,.5) {\small $i$};
    \end{tikzpicture}
    := (*_{0,n},\psi_i)
    \hspace{1cm}
    \begin{tikzpicture}
        \node[int] (v) at (0,0) {};
        \node[int] (w) at (0.5,0) {};
        \draw (v) edge +(-.5,-.5) edge +(-.5,.5) edge[dashed] +(-.5,.2) edge[dashed] +(-.2,.5) edge[dashed] +(-.5,-.2) edge[dashed] +(-.2,-.5) edge[crossed] (w)
        (w) edge +(.5,-.5) edge +(.5,.5) edge[dashed] +(.5,.2) edge[dashed] +(.5,-.2) edge[dashed] +(.2,.5) edge[dashed] +(.2,-.5);
        \draw [decorate,decoration={brace,amplitude=5pt,mirror}]
        (-.7,.5) -- (-.7,-.5) node[midway,xshift=-1em]{$A$};
        \draw [decorate,decoration={brace,amplitude=5pt}]
        (1.1,.5) -- (1.1,-.5) node[midway,xshift=1em]{$A'$};
    \end{tikzpicture}
    := (*_{0,n},\delta\{\substack{A \\ A'}\})
\]
There are two equivalent families that generate the relations in $H^{1,1}(\MM_{0,n})$. For any three pairwise distinct $i,x,y\in N$, or for any $i\neq j\in N$, it holds: 
\begin{equation} \label{equ:Relations2}
    \psi_i = \sum_{\substack{A\sqcup A'=N,|A|,|A'|\geq 2 \\ i\in A,x,y\in A'}} \delta\{\substack{A \\ A'}\}  \hspace{2cm}  \psi_i+\psi_j = \sum_{\substack{A\sqcup A'=N,|A|,|A'|\geq 2 \\ i\in A,j\in A'}} \delta\{\substack{A \\ A'}\}.
\end{equation}
So the $\psi_i$ classes are actually superfluous in the case $g=0$, but algebraically they can be more convenient to work with. The dimension of $H^{1,1}(\MM_{0,n})$ turns out to be $2^{n-1}- \binom{n}{2} -1$, in particular when $n=4$ any non-zero class forms a basis.

\subsubsection{The case $k=(11,0)$, $g=1$} This case is studied in \cite[Section 2]{CLP}. $H^{11,0}(\MM_{1,n})$ is non zero only for $n\geq 11$; so we operate under this assumption. The group is generated by classes $\omega_B$ for every ordered subset $B\subseteq N$ with $|B|=11$, where any two orderings are equivalent up to the sign of the permutation of $B$ that relates them. This means that the underlying set $B$ determines $\omega_B$ up to sign. Thus, if we choose a canonical ordering on $N$, we can stipulate that every subset comes equipped with the unique increasing ordering. We draw arrows onto the hairs contained in $B$ to depict the $\omega_B$ decoration.
\[
\begin{tikzpicture}
    \node[ext] (v) at (0,0) {};
    \node at (-.62, .1) {$\scriptscriptstyle \vdots$};
    \draw (v) edge[->-] +(-.6,-.5) edge[->-] +(-.6,-.3) edge[->-] +(-.3,-.5)
    edge[->-] +(-.6,.5) edge[->-] +(-.6,.3) edge[->-] +(-.3,.5) 
    edge[dashed] +(.6,.5) edge[dashed] +(.6,+.3)  edge[dashed] +(.3,+.5)
    edge[dashed] +(.6,-.5) edge[dashed] +(.6,-.3) edge[dashed] +(.3,-.5) ;
    \draw [decorate,decoration={brace,amplitude=5pt,mirror}]
  (-.8,.5) -- (-.8,-.5) node[midway,xshift=-1em]{$B$};
    \draw [decorate,decoration={brace,amplitude=5pt}]
    (.7,.5) -- (.7,-.5) node[midway,xshift=1em]{$B^c$};
\end{tikzpicture}
:= (*_{1,n},\omega_B)
\]
A generating set of relations is amongst the classes $\omega_B$ whose set $B$ is contained in the same size $12$ subset. Namely, if we choose a canonical ordering on $N$, then for every $E=\{e_1,...,e_{12}\}\subseteq N$ with $e_i$ increasing there is the relation
\begin{equation} \label{equ:Relations11}
    \sum_{i=1}^{12} (-1)^i \omega_{E\setminus e_i} =0
\end{equation}
Therefore, choosing a distinguished hair $e\in N$ (for example $e=1$), the classes $\omega_B$ with $e\in B$ form a basis of $H^{11,0}(\MM_{1,n})$.

The $\ss_n$-action on $\omega_B$ is given by $\sigma\,\omega_B=\omega_{\sigma B}$, which is equal to $\sgn \sigma \,\omega_B$ if $\sigma$ preserves $B$ setwise.

\subsubsection{The case $k=(12,1)$, $g=1$} This case is studied in \cite[Section 4.2]{CLPW}.
$H^{12,1}(\MM_{1,n})$ is non zero only for $n\geq 12$; so we operate under this assumption. The group is generated by classes $Z_{B\subseteq A}$ for every subset $A\subseteq N$ with $|A^c|\geq 2$ and ordered subset $B\subseteq A$ with $|B|=10$, where any two orderings are equivalent up to the sign of the permutation of $B$ that relates them. We draw arrows onto the hairs contained in $B$, and we split symbolically the vertex in a genus $1$ and a genus $0$ vertex, where we attach the hairs in $A$ and hairs in $A^c$ respectively.
\[
\begin{tikzpicture}
    \node[ext] (v) at (0,0) {};
    \node at (-.62, .1) {$\scriptscriptstyle \vdots$};
    \node[int] (w) at (.8, 0) {};
    \draw (v) edge[->-] +(-.6,-.5) edge[->-] +(-.6,-.3) edge[->-] +(-.3,-.5)
    edge[->-] +(-.6,.5) edge[->-] +(-.6,.3) edge[->-] +(-.3,.5) 
    edge[dashed] +(.6,.5) edge[dashed] +(.6,+.3)  edge[dashed] +(.3,+.5)
    edge[dashed] +(.6,-.5) edge[dashed] +(.6,-.3) edge[dashed] +(.3,-.5) edge[crossed] (w);
    \draw (w) edge +(.6,.5) edge[dashed] +(.6,+.3)  edge[dashed] +(.3,+.5)
    edge +(.6,-.5) edge[dashed] +(.6,-.3) edge[dashed] +(.3,-.5);
    \draw [decorate,decoration={brace,amplitude=5pt,mirror}]
    (-.8,.5) -- (-.8,-.5) node[midway,xshift=-1em]{$B$};
    \draw [decorate,decoration={brace,amplitude=5pt}]
    (1.5,.5) -- (1.5,-.5) node[midway,xshift=1em]{$A^c$};
    \draw [decorate,decoration={brace,mirror,amplitude=5pt}]
    (-.7,-.6) -- (.7,-.6) node[midway,yshift=-0.8em]{\small $A$};
\end{tikzpicture}
:= (*_{1,n},Z_{B\subseteq A})
\]
A generating set of relations is amongst the classes $Z_{B\subseteq A}$ having $|A^c|=2$ and same set $B\sqcup A^c$. Namely, if we choose a canonical ordering on $N$, then for every $E=\{e_1,...,e_{12}\}\subseteq N$ with $e_i$ increasing and every $1\leq i<j<k\leq 12$ we have the relation
\begin{equation} \label{equ:Relations13}
    (-1)^{i+j}Z_{E\setminus e_i,e_j \subseteq N\setminus e_i,e_j} - (-1)^{i+k}Z_{E\setminus e_i,e_k\subseteq N\setminus e_i,e_k} + (-1)^{j+k}Z_{E\setminus e_j,e_k\subseteq N\setminus e_j,e_k} = 0
\end{equation}
For this subset $E\subseteq N$, choosing a distinguished element $\tilde{e}\in E$ (for example $\tilde{e}=e_1$), the subspace $H_{A_2}^E$ spanned by the classes $Z_{B\subseteq A}$ with $B\sqcup A^c=E$ has basis the ones with $\tilde{e}\in A^c$, of which there are $11$. So $H^{12,1}(\MM_{1,n})$ decomposes into a direct sum $H_{A_3} \oplus \bigoplus_{|E|=12} H_{A_2}^E$, where $H_{A_3}$ has basis the classes with $|A^c|\geq 3$.

The $\ss_n$-action on $Z_{B\subseteq A}$ is given by $\sigma\,Z_{B\subseteq A}=Z_{\sigma B\subseteq \sigma A}$, which is equal to $\sgn \sigma \,Z_{B\subseteq \sigma A}$ if $\sigma$ preserves $B$ setwise.

\subsection{Action of the differential on the generators} \label{subsec:differential}
The differential $d$ of the graph complex $\GK^{\kappa}_{g,n}$, acts by summing over every vertex and every way of splitting that vertex with its decoration \cite[Section 2.6]{CLPW2}.

Let us first describe the image $d((*_{g,n},\gamma))$ of one-vertex decorated graphs $(*_{g,n},\gamma)$ under $d$. Every possible splitting is of the form $(*'-*'',\gamma'\otimes\gamma'')$, where $*'-*''$ is the connection of two vertices $*'_{g',n'},*''_{g'',n''}$ with $g=g'+g''$, $n=n'+n''+2$, and $\gamma',\gamma''$ are their respective decorations.
The hairs of $*'$ and $*''$ form a partition of $N$, the set of hairs of $*_{g,n}$. $\gamma'$ (and similarly $\gamma''$) is obtained from $\gamma$ by pullback along the map $\MM_{g',n'+1}\rightarrow \MM_{g,n}$ determined by the subset of $N$ ending up on $*'$. For the computation of the pullbacks of cohomology classes we refer to \cite{PayneWillwacher21},\cite{CLP} and \cite{CLPW} for the weight $2$, weight $11$ and weight $13$ cases respectively, in this paper we limit ourselves to translating those computations in graphical form.\\
The image under the differential is given by summing over these two-vertex decorated graphs for all possible splittings. In the case of interest, we always have either $g'=g''=g=0$ or $g'=g=1,g''=0$, so the only determining datum of the splitting is the partition of the hairs of $*_{g,n}$.

\begin{equation} \label{equ:Diff1}
\begin{tikzpicture}
    \node[int] (v) at (0,0) {};
    \draw (v) edge[dashed] +(.4,0)  edge +(-.4,.4) edge +(.4,.4);
    \draw (v) edge[dashed] +(-.4,0) edge +(-.4,-.4)  edge +(.4,-.4);
\end{tikzpicture}
\hspace{.3in}
\xlongrightarrow{d}
\hspace{.3in}
\sum_{\substack{S\sqcup S' = N \\ |S|,|S'|\geq 2}}
\begin{tikzpicture}
    \node[int] (v) at (0,0) {};
    \node[int] (w) at (.6,0) {};
    \draw (v) edge[dashed] +(-.4,0)  edge +(-.4,.4) edge +(-.4,-.4) edge (w);
    \draw (w) edge[dashed] +(.4,0) edge +(.4,.4)  edge +(.4,-.4);
    \draw [decorate,decoration={brace,amplitude=5pt,mirror}]
    (-.5,.5) -- (-.5,-.5) node[midway,xshift=-1em]{$S$};
    \draw [decorate,decoration={brace,amplitude=5pt}]
    (1.1,.5) -- (1.1,-.5) node[midway,xshift=1em]{$S'$};
\end{tikzpicture}
\end{equation}

\begin{equation} \label{equ:Diff2irr}
    \begin{tikzpicture}
        \node[int] (v) at (0,0) {};
        \draw (v) edge +(0,-.3) edge[loop, crossed, distance=0.8cm] (v);
    \end{tikzpicture}
\hspace{.1in}
\xlongrightarrow{d}
\hspace{.1in}
0 \hspace{.3in} \text{because the valence of $*_{1,1}$ is less than $2$}
\end{equation}

\begin{multline} \label{equ:Diff2}
\text{For any choice of $x,y\in A$ and $x',y'\in A'$, the image can be expressed as follows:}\\
\begin{tikzpicture}
    \node[int] (v) at (0,0) {};
    \node[int] (w) at (0.5,0) {};
    \draw (v) edge +(-.5,-.5) edge +(-.5,.5) edge[dashed] +(-.5,.2) edge[dashed] +(-.2,.5) edge[dashed] +(-.5,-.2) edge[dashed] +(-.2,-.5) edge[crossed] (w)
    (w) edge +(.5,-.5) edge +(.5,.5) edge[dashed] +(.5,.2) edge[dashed] +(.5,-.2) edge[dashed] +(.2,.5) edge[dashed] +(.2,-.5);
    \draw [decorate,decoration={brace,amplitude=5pt,mirror}]
    (-.7,.5) -- (-.7,-.5) node[midway,xshift=-1em]{$A$};
    \draw [decorate,decoration={brace,amplitude=5pt}]
    (1.1,.5) -- (1.1,-.5) node[midway,xshift=1em]{$A'$};
\end{tikzpicture}
\hspace{.05in}
\xlongrightarrow{d}
\hspace{.05in}
-\sum_{\substack{\tilde{A}\subset A \\ x,y\in\tilde{A}}}
\begin{tikzpicture}
    \node[int] (v) at (0,0) {};
    \node[int] (w) at (0.5,0) {};
    \node[int] (z) at (1,0) {};
    \node[] (x) at (-.6,.5) {\tiny $x$};
    \node[] (y) at (-.6,.2) {\tiny $y$};
    \draw (v)  edge +(-.5,.5) edge +(-.5,.2) edge[dashed] +(-.2,.5)  edge[crossed] (w)
    (w) edge (z) edge +(-.5,-.5) edge[dashed] +(-.5,-.3) edge[dashed] +(-.2,-.5);
    \draw (z) edge +(.5,-.5) edge +(.5,.5) edge[dashed] +(.5,.2) edge[dashed] +(.5,-.2) edge[dashed] +(.2,.5) edge[dashed] +(.2,-.5);
    \draw [decorate,decoration={brace,amplitude=5pt,mirror}]
    (-.8,.6) -- (-.8,0) node[midway,xshift=-1em]{\tiny $\tilde{A}$};
    \draw [decorate,decoration={brace,amplitude=5pt,mirror}]
    (-.2,0) -- (-.2,-.5) node[midway,xshift=-1.5em]{\tiny $A\setminus\tilde{A}$};
    \draw [decorate,decoration={brace,amplitude=5pt}]
    (1.6,.5) -- (1.6,-.5) node[midway,xshift=1em]{$A'$};
\end{tikzpicture}
-\sum_{\substack{\tilde{A}\subset A' \\ x,y\in\tilde{A}}}
\begin{tikzpicture}
    \node[int] (z) at (-1,0) {};
    \node[int] (v) at (-0.5,0) {};
    \node[int] (w) at (0,0) {};
    \node[] (x) at (.7,.5) {\tiny $x'$};
    \node[] (y) at (.7,.2) {\tiny $y'$};
    \draw (z) edge +(-.5,-.5) edge +(-.5,.5) edge[dashed] +(-.5,.2) edge[dashed] +(-.2,.5) edge[dashed] +(-.5,-.2) edge[dashed] +(-.2,-.5) edge (v);
    \draw (v) edge[crossed] (w) edge +(.5,-.5) edge[dashed] +(.5,-.3) edge[dashed] +(.2,-.5);
    \draw (w) edge +(.5,.5) edge +(.5,.2) edge[dashed] +(.2,.5) ;
    \draw [decorate,decoration={brace,amplitude=5pt,mirror}]
    (-1.7,.5) -- (-1.7,-.5) node[midway,xshift=-1em]{$A$};
    \draw [decorate,decoration={brace,amplitude=5pt}]
    (.9,.6) -- (.9,0) node[midway,xshift=1em]{\tiny $\tilde{A}$};
    \draw [decorate,decoration={brace,amplitude=5pt}]
    (.2,0) -- (.2,-.5) node[midway,xshift=1.5em]{\tiny $A'\setminus\tilde{A}$};
\end{tikzpicture}\\
+\sum_{\substack{S\subset A' \\ |S|\geq 2}}
\begin{tikzpicture}
    \node[int] (z) at (-1,0) {};
    \node[int] (v) at (-0.5,0) {};
    \node[int] (w) at (0,0) {};
    \draw (z) edge +(-.5,-.5) edge +(-.5,.5) edge[dashed] +(-.5,.2) edge[dashed] +(-.2,.5) edge[dashed] +(-.5,-.2) edge[dashed] +(-.2,-.5) edge[crossed] (v);
    \draw (v) edge (w) edge +(.5,-.5) edge[dashed] +(.5,-.3) edge[dashed] +(.2,-.5);
    \draw (w) edge +(.5,.5) edge +(.5,.2) edge[dashed] +(.2,.5) ;
    \draw [decorate,decoration={brace,amplitude=5pt,mirror}]
    (-1.7,.5) -- (-1.7,-.5) node[midway,xshift=-1em]{$A$};
    \draw [decorate,decoration={brace,amplitude=5pt}]
    (.6,.6) -- (.6,0) node[midway,xshift=1em]{\tiny $S$};
    \draw [decorate,decoration={brace,amplitude=5pt}]
    (.2,0) -- (.2,-.5) node[midway,xshift=1.5em]{\tiny $A'\setminus S$};
\end{tikzpicture}
+\sum_{\substack{S\subset A \\ |S|\geq 2}}
\begin{tikzpicture}
    \node[int] (v) at (0,0) {};
    \node[int] (w) at (0.5,0) {};
    \node[int] (z) at (1,0) {};
    \draw (v)  edge +(-.5,.5) edge +(-.5,.2) edge[dashed] +(-.2,.5)  edge[crossed] (w)
    (w) edge (z) edge +(-.5,-.5) edge[dashed] +(-.5,-.3) edge[dashed] +(-.2,-.5);
    \draw (z) edge +(.5,-.5) edge +(.5,.5) edge[dashed] +(.5,.2) edge[dashed] +(.5,-.2) edge[dashed] +(.2,.5) edge[dashed] +(.2,-.5);
    \draw [decorate,decoration={brace,amplitude=5pt,mirror}]
    (-.6,.6) -- (-.6,0) node[midway,xshift=-1em]{\tiny $S$};
    \draw [decorate,decoration={brace,amplitude=5pt,mirror}]
    (-.2,0) -- (-.2,-.5) node[midway,xshift=-1.5em]{\tiny $A\setminus S$};
    \draw [decorate,decoration={brace,amplitude=5pt}]
    (1.6,.5) -- (1.6,-.5) node[midway,xshift=1em]{$A'$};
\end{tikzpicture}
\end{multline}
\vspace{-.02in}
In the second term of \eqref{equ:Diff11}, the weight $11$ decoration $\omega_B$ becomes $\omega_{B\setminus\tilde{b}\sqcup q}$, where $q$ is the newly added half-edge to the genus $1$ vertex and takes the place of $\tilde{b}$ in the ordering of $B$.
\vspace{-.02in}
\begin{equation} \label{equ:Diff11}
    \begin{tikzpicture}
        \node[ext] (v) at (0,0) {};
        \node at (-.62, .1) {$\scriptscriptstyle \vdots$};
        \draw (v) edge[->-] +(-.6,-.5) edge[->-] +(-.6,-.3) edge[->-] +(-.3,-.5)
        edge[->-] +(-.6,.5) edge[->-] +(-.6,.3) edge[->-] +(-.3,.5) 
        edge[dashed] +(.6,.5) edge[dashed] +(.6,+.3)  edge[dashed] +(.3,+.5)
        edge[dashed] +(.6,-.5) edge[dashed] +(.6,-.3) edge[dashed] +(.3,-.5) ;
        \draw [decorate,decoration={brace,amplitude=5pt,mirror}]
      (-.8,.5) -- (-.8,-.5) node[midway,xshift=-1em]{$B$};
        \draw [decorate,decoration={brace,amplitude=5pt}]
        (.7,.5) -- (.7,-.5) node[midway,xshift=1em]{$B^c$};
    \end{tikzpicture}
    \hspace{.05in}
    \xlongrightarrow{d}
    \hspace{.05in}
    \sum_{\substack{S\subseteq B^c \\ |S|\geq 2}}
    \begin{tikzpicture}
        \node[ext] (v) at (0,0) {};
        \node at (-.62, .1) {$\scriptscriptstyle \vdots$};
        \node[int] (z) at (.5,0) {};
        \draw (v) edge[->-] +(-.6,-.5) edge[->-] +(-.6,-.3) edge[->-] +(-.3,-.5)
        edge[->-] +(-.6,.5) edge[->-] +(-.6,.3) edge[->-] +(-.3,.5) 
        edge[dashed] +(.6,-.5) edge[dashed] +(.6,-.3) edge[dashed] +(.3,-.5) edge (z);
        \draw (z) edge +(.6,.5) edge +(.6,+.3)  edge[dashed] +(.3,+.5);
        \draw [decorate,decoration={brace,amplitude=5pt,mirror}]
      (-.8,.5) -- (-.8,-.5) node[midway,xshift=-1em]{$B$};
        \draw [decorate,decoration={brace,amplitude=5pt}]
        (1.2,.6) -- (1.2,0) node[midway,xshift=1em]{\tiny $S$};
        \draw [decorate,decoration={brace,amplitude=5pt}]
        (.7,0) -- (.7,-.5) node[midway,xshift=1.5em]{\tiny $B^c\setminus S$};
    \end{tikzpicture}
    +\sum_{\substack{\varnothing\neq S\subseteq B^c \\ \tilde{b}\in B}}
    \begin{tikzpicture}
        \node[ext] (v) at (0,0) {};
        \node at (-.62, .1) {$\scriptscriptstyle \vdots$};
        \node[int] (z) at (.7,0) {};
        \node (b) at (1.4,.6) {\tiny $\tilde{b}$};
        \draw (v) edge[->-] +(-.6,-.3) edge[->-] +(-.6,-.5) edge[->-] +(-.3,-.5)
        edge[->-] +(-.6,.3) edge[->-] +(-.3,.5) 
        edge[dashed] +(.6,-.5) edge[dashed] +(.6,-.3) edge[dashed] +(.3,-.5) edge[->-] (z);
        \draw (z) edge +(.6,.5) edge +(.6,+.3)  edge[dashed] +(.3,+.5);
        \draw [decorate,decoration={brace,amplitude=5pt,mirror}]
      (-.7,.5) -- (-.7,-.5) node[midway,xshift=-1.3em]{\tiny $B\setminus\tilde{b}$};
        \draw [decorate,decoration={brace,amplitude=5pt}]
        (1.4,.35) -- (1.4,0) node[midway,xshift=1em]{\tiny $S$};
        \draw [decorate,decoration={brace,amplitude=5pt}]
        (.9,0) -- (.9,-.5) node[midway,xshift=1.5em]{\tiny $B^c\setminus S$};
    \end{tikzpicture}
\end{equation}
\vspace{-.02in}
For any fixed choice of $x,y\in A^c$, the image of a weight $13$ decoration can be expressed as follows. In the two terms that create weight $11$ and $2$ vertices, the newly created weight $11$ decoration $\omega_{B\sqcup p}$, where $p$ is the half-edge at the genus $1$ vertex, is understood to have the ordering inherited from $B$ with $p$ appended at the end.
\vspace{-.04in}
\begin{multline} \label{equ:Diff12}
\begin{tikzpicture}
    \node[ext] (v) at (0,0) {};
    \node at (-.62, .1) {$\scriptscriptstyle \vdots$};
    \node[int] (w) at (.8, 0) {};
    \draw (v) edge[->-] +(-.6,-.5) edge[->-] +(-.6,-.3) edge[->-] +(-.3,-.5)
    edge[->-] +(-.6,.5) edge[->-] +(-.6,.3) edge[->-] +(-.3,.5) 
    edge[dashed] +(.6,.5) edge[dashed] +(.6,+.3)  edge[dashed] +(.3,+.5)
    edge[dashed] +(.6,-.5) edge[dashed] +(.6,-.3) edge[dashed] +(.3,-.5) edge[crossed] (w);
    \draw (w) edge +(.6,.5) edge[dashed] +(.6,+.3)  edge[dashed] +(.3,+.5)
    edge +(.6,-.5) edge[dashed] +(.6,-.3) edge[dashed] +(.3,-.5);
    \draw [decorate,decoration={brace,amplitude=5pt,mirror}]
    (-.8,.5) -- (-.8,-.5) node[midway,xshift=-1em]{$B$};
    \draw [decorate,decoration={brace,amplitude=5pt}]
    (1.5,.5) -- (1.5,-.5) node[midway,xshift=1em]{$A^c$};
\end{tikzpicture}
\xlongrightarrow{d}
\sum_{\substack{\tilde{S}\subseteq A\setminus B \\ |\tilde{S}|\geq 2}}
\begin{tikzpicture}
    \node[ext] (v) at (0,0) {};
    \node at (-.62, .1) {$\scriptscriptstyle \vdots$};
    \node[int] (w) at (.8, 0) {};
    \node[int] (z) at (.2,.5) {};
    \draw (v) edge (z) edge[->-] +(-.6,-.5) edge[->-] +(-.6,-.3) edge[->-] +(-.3,-.5)
    edge[->-] +(-.6,.5) edge[->-] +(-.6,.3) edge[->-] +(-.3,.5) 
    edge[dashed] +(.6,-.5) edge[dashed] +(.6,-.3) edge[dashed] +(.3,-.5) edge[crossed] (w);
    \draw (z) edge +(.6,.5) edge +(.6,+.3)  edge[dashed] +(.3,+.5);
    \draw (w) edge +(.6,.5) edge[dashed] +(.6,+.3)  edge[dashed] +(.3,+.5)
    edge +(.6,-.5) edge[dashed] +(.6,-.3) edge[dashed] +(.3,-.5);
    \draw [decorate,decoration={brace,amplitude=5pt,mirror}]
    (-.8,.5) -- (-.8,-.5) node[midway,xshift=-1em]{$B$};
    \draw [decorate,decoration={brace,amplitude=5pt}]
    (.9,1.1) -- (.9,.6) node[midway,xshift=1em]{\tiny $\tilde{S}$};
    \draw [decorate,decoration={brace,amplitude=5pt}]
    (1.5,.5) -- (1.5,-.5) node[midway,xshift=1em]{$A^c$};
\end{tikzpicture}
+\sum_{\substack{\varnothing\neq\tilde{S}\subseteq A\setminus B \\ \tilde{b}\in B}}
\hspace{-.06in}
\begin{tikzpicture}
    \node[ext] (v) at (0,0) {};
    \node at (-.62, .1) {$\scriptscriptstyle \vdots$};
    \node[int] (w) at (.8, 0) {};
    \node[int] (z) at (.2,.5) {};
    \node (b) at (-.3,.8) {\tiny $\tilde{b}$};
    \draw (v) edge[->-] (z) edge[->-] +(-.6,-.5) edge[->-] +(-.6,-.3) edge[->-] +(-.3,-.5)
    edge[->-] +(-.6,.5) edge[->-] +(-.6,.3)
    edge[dashed] +(.6,-.5) edge[dashed] +(.6,-.3) edge[dashed] +(.3,-.5) edge[crossed] (w);
    \draw (z) edge +(.6,.5) edge[dashed] +(.6,+.3)  edge[dashed] +(.3,+.5) edge +(-.4,.5);
    \draw (w) edge +(.6,.5) edge[dashed] +(.6,+.3)  edge[dashed] +(.3,+.5)
    edge +(.6,-.5) edge[dashed] +(.6,-.3) edge[dashed] +(.3,-.5);
    \draw [decorate,decoration={brace,amplitude=5pt,mirror}]
    (-.7,.5) -- (-.7,-.5) node[midway,xshift=-1.3em]{\tiny$B\setminus\tilde{b}$};
    \draw [decorate,decoration={brace,amplitude=5pt}]
    (.9,1.1) -- (.9,.6) node[midway,xshift=1em]{\tiny $\tilde{S}$};
    \draw [decorate,decoration={brace,amplitude=5pt}]
    (1.5,.5) -- (1.5,-.5) node[midway,xshift=1em]{$A^c$};
\end{tikzpicture} \\
\sum_{\substack{S\subset A^c \\ |S|\geq 2}}
\begin{tikzpicture}
    \node[ext] (v) at (0,0) {};
    \node at (-.62, .1) {$\scriptscriptstyle \vdots$};
    \node[int] (w) at (.8, 0) {};
    \node[int] (z) at (1.3,0) {};
    \draw (v) edge[->-] +(-.6,-.5) edge[->-] +(-.6,-.3) edge[->-] +(-.3,-.5)
    edge[->-] +(-.6,.5) edge[->-] +(-.6,.3) edge[->-] +(-.3,.5) 
    edge[dashed] +(.6,.5) edge[dashed] +(.6,+.3)  edge[dashed] +(.3,+.5)
    edge[dashed] +(.6,-.5) edge[dashed] +(.6,-.3) edge[dashed] +(.3,-.5) edge[crossed] (w);
    \draw (w) edge +(.6,-.5) edge[dashed] +(.6,-.3) edge[dashed] +(.3,-.5) edge (z);
    \draw (z) edge +(.6,.5) edge +(.6,+.3)  edge[dashed] +(.3,+.5);
    \draw [decorate,decoration={brace,amplitude=5pt,mirror}]
    (-.8,.5) -- (-.8,-.5) node[midway,xshift=-1em]{$B$};
    \draw [decorate,decoration={brace,amplitude=5pt}]
    (2,.6) -- (2,0) node[midway,xshift=1em]{\tiny $S$};
    \draw [decorate,decoration={brace,amplitude=5pt}]
    (1.5,0) -- (1.5,-.5) node[midway,xshift=1.5em]{\tiny $A^c\setminus S$};
\end{tikzpicture}
+\sum_{\varnothing\neq\tilde{S}\subseteq A\setminus B}
\begin{tikzpicture}
    \node[ext] (v) at (0,0) {};
    \node at (-.62, .1) {$\scriptscriptstyle \vdots$};
    \node[int] (w) at (.7, 0) {};
    \node[int] (z) at (1.4,0) {};
    \draw (v) edge[->-] +(-.6,-.5) edge[->-] +(-.6,-.3) edge[->-] +(-.3,-.5)
    edge[->-] +(-.6,.5) edge[->-] +(-.6,.3) edge[->-] +(-.3,.5) 
    edge[dashed] +(.6,-.5) edge[dashed] +(.6,-.3) edge[dashed] +(.3,-.5) edge[->-] (w);
    \draw (w)  edge[crossed] (z) edge +(.4,.5) edge[dashed] +(.6,+.4)  edge[dashed] +(.2,+.5);
    \draw (z) edge +(.6,.5) edge[dashed] +(.6,+.3)  edge[dashed] +(.3,+.5)
    edge +(.6,-.5) edge[dashed] +(.6,-.3) edge[dashed] +(.3,-.5);
    \draw [decorate,decoration={brace,amplitude=5pt,mirror}]
    (-.8,.5) -- (-.8,-.5) node[midway,xshift=-1em]{$B$};
    \draw [decorate,decoration={brace,amplitude=5pt}]
    (2.1,.5) -- (2.1,-.5) node[midway,xshift=1em]{$A^c$};
    \draw [decorate,decoration={brace,amplitude=5pt}]
    (.8,.5) -- (1.3,.5) node[midway,yshift=1em]{\tiny $\tilde{S}$};
\end{tikzpicture} \\
-\sum_{\varnothing\neq\tilde{S}\subseteq A\setminus B}
\begin{tikzpicture}
    \node[ext] (v) at (0,0) {};
    \node at (-.62, .1) {$\scriptscriptstyle \vdots$};
    \node[int] (w) at (.7, 0) {};
    \node[int] (z) at (1.4,0) {};
    \draw (v) edge[->-] +(-.6,-.5) edge[->-] +(-.6,-.3) edge[->-] +(-.3,-.5)
    edge[->-] +(-.6,.5) edge[->-] +(-.6,.3) edge[->-] +(-.3,.5) 
    edge[dashed] +(.6,-.5) edge[dashed] +(.6,-.3) edge[dashed] +(.3,-.5) edge[crossed] (w);
    \draw (w)  edge (z) edge +(.4,.5) edge[dashed] +(.6,+.4)  edge[dashed] +(.2,+.5);
    \draw (z) edge +(.6,.5) edge[dashed] +(.6,+.3)  edge[dashed] +(.3,+.5)
    edge +(.6,-.5) edge[dashed] +(.6,-.3) edge[dashed] +(.3,-.5);
    \draw [decorate,decoration={brace,amplitude=5pt,mirror}]
    (-.8,.5) -- (-.8,-.5) node[midway,xshift=-1em]{$B$};
    \draw [decorate,decoration={brace,amplitude=5pt}]
    (2.1,.5) -- (2.1,-.5) node[midway,xshift=1em]{$A^c$};
    \draw [decorate,decoration={brace,amplitude=5pt}]
    (.8,.5) -- (1.3,.5) node[midway,yshift=1em]{\tiny $\tilde{S}$};
\end{tikzpicture}
-\sum_{\substack{S\subset A^c  \\ x,y\in S}}
\begin{tikzpicture}
    \node[ext] (v) at (0,0) {};
    \node at (-.62, .1) {$\scriptscriptstyle \vdots$};
    \node[int] (w) at (.7, 0) {};
    \node[int] (z) at (1.3,0) {};
    \node (x) at (2,.5) {\tiny $x$};
    \node (y) at (2,.2) {\tiny $y$};
    \draw (v) edge[->-] +(-.6,-.5) edge[->-] +(-.6,-.3) edge[->-] +(-.3,-.5)
    edge[->-] +(-.6,.5) edge[->-] +(-.6,.3) edge[->-] +(-.3,.5) 
    edge[dashed] +(.6,.5) edge[dashed] +(.6,+.3)  edge[dashed] +(.3,+.5)
    edge[dashed] +(.6,-.5) edge[dashed] +(.6,-.3) edge[dashed] +(.3,-.5) edge[->-] (w);
    \draw (w) edge +(.6,-.5) edge[dashed] +(.6,-.3) edge[dashed] +(.3,-.5) edge[crossed] (z);
    \draw (z) edge +(.6,.5) edge +(.6,+.3)  edge[dashed] +(.3,+.5);
    \draw [decorate,decoration={brace,amplitude=5pt,mirror}]
    (-.8,.5) -- (-.8,-.5) node[midway,xshift=-1em]{$B$};
    \draw [decorate,decoration={brace,amplitude=5pt}]
    (2.2,.6) -- (2.1,0) node[midway,xshift=1em]{\tiny $S$};
    \draw [decorate,decoration={brace,amplitude=5pt}]
    (1.4,-.1) -- (1.4,-.6) node[midway,xshift=1.5em]{\tiny $A^c\setminus S$};
\end{tikzpicture}
\end{multline}
\vspace{-.03in}
We denote by $d_v(\Gamma,\gamma\otimes e)$ the decorated graph obtained by replacing the vertex $v$ and its decoration $\gamma_v$ by $d(*_{g,n},\gamma_v)$, gluing it's hairs to the hairs and neighbours of $v$ in $\Gamma$; the newly created edge is understood to be placed last in the ordering $e$. Thus, the total differential takes the form
\vspace{-.03in}
\[ \text{case (A) graphs: \;\;} (\Gamma,\gamma\otimes e) \xlongrightarrow{d} \hspace{.3cm} d_{\bar{v}}(\Gamma,\gamma\otimes e) \hspace{.1cm}  + \sum_{\bar{v}\neq v\in V(G)} d_v(\Gamma,\gamma\otimes e) \]
\[ \text{case (B) graphs: \;\;} (\Gamma,\gamma\otimes e) \xlongrightarrow{d} \hspace{.3cm} d_{\bar{v}}(\Gamma,\gamma\otimes e) \hspace{.1cm}  + d_{\tilde{v}}(\Gamma,\gamma\otimes e) +\sum_{\bar{v},\tilde{v}\neq v\in V(G)} d_v(\Gamma,\gamma\otimes e). \]

\subsection{The simplified graph complex in weight 13} \label{sec:SimplifiedGK}

The $\Q$-Hodge decompositions in \eqref{equ:GroupsHodgeDeco} give us $\GK_{g,n}^{13}\otimes\CC = \GK_{g,n}^{12,1} \oplus \GK_{g,n}^{1,12}$, where $\GK_{g,n}^{12,1}$ is obtained as in \eqref{equ:GK13} by replacing $H^{13}, H^{11}$ and $H^{2}$ by $H^{12,1}, H^{11,0}$ and $H^{1,1}$ respectively. After quotienting by an appropriate subspace closed under the differential, in \cite[Section 2.2]{CLPW2} they obtain a quasi-isomorphic complex $\GK_{g,n}^{13}\twoheadrightarrow \bGK_{g,n}^{13}$ generated by a much smaller set of decorated graphs. 
This quasi-isomorphism preserves the above $\Q$-Hodge decomposition: $\bGK_{g,n}^{13}\otimes\CC = \bGK_{g,n}^{12,1} \oplus \bGK_{g,n}^{1,12}$.
From now on we will focus on the $\bGK_{g,n}^{12,1}$ part, which we now describe.

Let $G\subseteq \GK_{g,n}^{12,1}$ be the subspace spanned by decorated graphs with the following properties:
\begin{enumerate}
    \item[1)] Weight zero vertices have genus zero, valence at least $3$ and don't have loops. There are no multiple edges, except possibly for edges incident at $\bar{v}$ or $\tilde{v}$.
    \item[2)] In both cases $(A)$ and $(B)$ the special vertex $\bar{v}$ has $g_{\bar{v}}=1$.
    \item[2b)] In case $(B)$ the special vertex $\tilde{v}$ has either $g_{\tilde{v}}=0$, or $g_{\tilde{v}}=1$ and $n_{\tilde{v}}=1$.
\end{enumerate}

Let $R\subseteq G$ be the subspace generated by the following relations:
\begin{enumerate}
    \item[3a)] Case $(A)$ graphs with a loop $(s,t)$ at $\bar{v}$ and decorated by a $Z_{B\subseteq A}$ with $|A^c|\geq 3$, $s,t\in A^c$ are set to zero.
    \item[3b)] Case $(B)$ graphs with $g_{\tilde{v}}=0$ and a loop at $\tilde{v}$ decorated by a class in the image of the pullback $H^2(\MM_{1,n_{\tilde{v}}-2})\xrightarrow{\xi_{*}} H^2(\MM_{0,n_{\tilde{v}}})$ are set to zero. 
    \item[4)] Case $(A)$ graphs with a loop $(s,t)$ at $\bar{v}$ and decorated by $Z_{B\subseteq A}$ with $A^c=\{s,t\}$ are identified with $\frac{1}{12}$ times the case $(B)$ graph obtained by adding a genus $1$ vertex $\tilde{v}$, connecting it only to $\bar{v}$ through an edge $(p',p)$, redecorating $\bar{v}$ with the weight $11$ class $\omega_{B\sqcup p}$ and decorating $\tilde{v}$ with the weight $2$ class $\delta_{irr}$.
\end{enumerate}

\begin{proposition} \label{bGKprojection}
    The projection $\GK_{g,n}^{12,1} \twoheadrightarrow G/R =: \bGK_{g,n}^{12,1}$ is a quasi-isomorphism.
\end{proposition}
\begin{proof}
    This is the content of \cite[Section 2.2]{CLPW2}. We have restricted to rewriting their relations $1-6,7',7''$ into the properties $1,2,2b$ and relations $3a,3b,4$. The absence of multiple edges in property $1)$ can be imposed because those graphs would be later killed by coinvariants.
\end{proof}

We expand \eqref{equ:GK13} to give an explicit formula for the subspace $G$ in the above proposition. After exchanging the order of the cohomology relations and the coinvariants we obtain
\begin{multline}\label{equ:G}
   G \hspace{.1in} = \hspace{.1in}
   \bigoplus_{\substack{[\Gamma],\bar{v},\,B\subseteq A\subseteq N_{\bar{v}} \\ |B|=10, |A^c|\geq 3}} \left( \left<Z_{B\subseteq A}\right> \otimes \Q[-1]^{\otimes |E(\Gamma)|} \right)/\text{Aut}_{\Gamma,\bar{v}}^{B,A} \\
    \oplus \bigoplus_{\substack{[\Gamma],\bar{v},\,E\subseteq N_{\bar{v}} \\ |E|=12}} \bigg( \bigoplus_{\substack{B\subseteq A\subseteq N_{\bar{v}} \\  B\sqcup A^c=E}} \left( \left<Z_{B\subseteq A}\right> \otimes \Q[-1]^{\otimes |E(\Gamma)|} \right)/\text{Aut}_{\Gamma,\bar{v}}^{B,A} \bigg) \Big/ \{\substack{\text{weight 13}\\\text{relations}}\} \\
    \oplus \bigoplus_{[\Gamma],\bar{v},\tilde{v}} \bigg( \bigoplus_{\substack{B\subseteq N_{\tilde{v}}, A\sqcup A'=N_{\tilde{v}}\\ |B|=11, |A|,|A'|\geq 2}} \left( \left<\omega_B\otimes\delta\{\substack{A \\ A'}\}\right> \otimes \Q[-1]^{\otimes |E(\Gamma)|}\right) / \text{Aut}_{\Gamma,\bar{v},\tilde{v}}^{B,\{A,A'\} } \bigg) \Big/ \{\substack{\text{weight 11 and 2}\\\text{relations}}\} \\
    \oplus \bigoplus_{[\Gamma],\bar{v},\tilde{v}} \bigg( \bigoplus_{\substack{B\subseteq N_{\tilde{v}}\\ |B|=11}} \left( \left<\omega_B\otimes \delta_{irr}\right> \otimes \Q[-1]^{\otimes |E(\Gamma)|}\right) / \text{Aut}_{\Gamma,\bar{v},\tilde{v}}^B \bigg) \Big/ \{\substack{\text{weight 11}\\\text{relations}}\} ,
\end{multline}
where, for example in the third term, Aut$_{\Gamma,\bar{v},\tilde{v}}^{B,\{A,A'\}}$ is the group of automorphisms leaving the set $B$ and the partition $A\sqcup A'$ invariant; in all four direct sums the isomorphism classes $[\Gamma]$ with special vertices $\bar{v},\tilde{v}$ are understood run only over graphs satisfying properties $1,2,2b$.

The automorphism groups in \eqref{equ:G} are precisely the symmetries of the decorated graphs $(\Gamma,Z_{B\subseteq A}\otimes e)$, $(\Gamma,\omega_B\otimes\delta\{\substack{A \\ A'}\}\otimes e)$ and $(\Gamma,\omega_B\otimes \delta_{irr}\otimes e)$ respectively, which act by the sign of the permutation of the half-edges of $\bar{v}$ in $B$ multiplied by the sign of the permutation on the internal edges $E(\Gamma)$. In each case, the coinvariants restrict to killing the decorated graphs with odd symmetry.

Using the graphical depictions of cohomology classes, the relations in \ref{bGKprojection} look as follows.
\[
    \text{3b):   }
    \begin{tikzpicture}
        \node[int] (v) at (0,0) {};
        \node[int] (w) at (0.5,0) {};
        \node (s) at (-.6,.5)  {\tiny$s$};
        \node (t) at (-.6,-.5)  {\tiny$t$};
        \draw (v) to [out=135,in=225,distance=1.3cm] (v) edge[dashed] +(-.5,.2) edge[dashed] +(-.2,.5) edge[dashed] +(-.5,-.2) edge[dashed] +(-.2,-.5) edge[crossed] (w)
        (w) edge +(.5,-.5) edge +(.5,.5) edge[dashed] +(.5,.2) edge[dashed] +(.5,-.2) edge[dashed] +(.2,.5) edge[dashed] +(.2,-.5);
        \draw [decorate,decoration={brace,amplitude=5pt,mirror}]
        (-.8,.5) -- (-.8,-.5) node[midway,xshift=-1em]{$A$};
        \draw [decorate,decoration={brace,amplitude=5pt}]
        (1.1,.5) -- (1.1,-.5) node[midway,xshift=1em]{$A'$};
    \end{tikzpicture}
    = 0
    \hspace{1cm}
    \sum_{\substack{A\sqcup A'= N_{\tilde{v}}\\ |A|,|A'|\geq 2 \\ s\in A, t\in A'}}
    \begin{tikzpicture}
        \node[int] (v) at (0,0) {};
        \node[int] (w) at (0.5,0) {};
        \node (s) at (-.1,.5)  {\tiny$s$};
        \node (t) at (.6,.5)  {\tiny$t$};
        \draw (v) edge +(-.5,-.5) edge[dashed] +(-.5,.5) edge[dashed] +(-.5,.2) edge[dashed] +(-.5,-.2) edge[dashed] +(-.2,-.5) edge[crossed] (w) edge[bend left=90,distance=.8cm] (w)
        (w) edge +(.5,-.5) edge[dashed] +(.5,.5) edge[dashed] +(.5,.2) edge[dashed] +(.5,-.2) edge[dashed] +(.2,-.5);
        \draw [decorate,decoration={brace,amplitude=5pt,mirror}]
        (-.7,.5) -- (-.7,-.5) node[midway,xshift=-1em]{$A$};
        \draw [decorate,decoration={brace,amplitude=5pt}]
        (1.1,.5) -- (1.1,-.5) node[midway,xshift=1em]{$A'$};
    \end{tikzpicture}
    =0
\]

\[
    \text{3a):  }
    \begin{tikzpicture}
        \node[ext] (v) at (0,0) {};
        \node at (-.62, .1) {$\scriptscriptstyle \vdots$};
        \node[int] (w) at (.8, 0) {};
        \node (s) at (1.5,.4) {\tiny$s$};
        \node (t) at (1.5,-.4) {\tiny$t$};
        \draw (v) edge[->-] +(-.6,-.5) edge[->-] +(-.6,-.3) edge[->-] +(-.3,-.5)
        edge[->-] +(-.6,.5) edge[->-] +(-.6,.3) edge[->-] +(-.3,.5) 
        edge[dashed] +(.6,.5) edge[dashed] +(.6,+.3)  edge[dashed] +(.3,+.5)
        edge[dashed] +(.6,-.5) edge[dashed] +(.6,-.3) edge[dashed] +(.3,-.5) edge[crossed] (w);
        \draw (w)  to [out=35,in=-35,distance=1cm] (w) edge[dashed] +(.6,+.5)  edge +(.3,+.5)
        edge[dashed] +(.6,-.5) edge[dashed] +(.3,-.5);
        \draw [decorate,decoration={brace,amplitude=5pt,mirror}]
        (-.8,.5) -- (-.8,-.5) node[midway,xshift=-1em]{$B$};
        \draw [decorate,decoration={brace,amplitude=5pt}]
        (1.6,.5) -- (1.6,-.5) node[midway,xshift=1em]{$A^c$};
        \draw [decorate,decoration={brace,mirror,amplitude=5pt}]
        (-.7,-.6) -- (.7,-.6) node[midway,yshift=-0.8em]{\small $A$};
    \end{tikzpicture}
    =0
    \hspace{1cm}
    \text{4):  }
    \begin{tikzpicture}
        \node[ext] (v) at (0,0) {};
        \node at (-.62, .1) {$\scriptscriptstyle \vdots$};
        \node[int] (w) at (.8, 0) {};
        \node (s) at (1.2,.4) {\tiny$s$};
        \node (t) at (1.2,-.4) {\tiny$t$};
        \draw (v) edge[->-] +(-.6,-.5) edge[->-] +(-.6,-.3) edge[->-] +(-.3,-.5)
        edge[->-] +(-.6,.5) edge[->-] +(-.6,.3) edge[->-] +(-.3,.5) 
        edge[dashed] +(.6,.5) edge[dashed] +(.6,+.3)  edge[dashed] +(.3,+.5)
        edge[dashed] +(.6,-.5) edge[dashed] +(.6,-.3) edge[dashed] +(.3,-.5) edge[crossed] (w);
        \draw (w)  to [out=35,in=-35,distance=1cm] (w);
        \draw [decorate,decoration={brace,amplitude=5pt,mirror}]
        (-.8,.5) -- (-.8,-.5) node[midway,xshift=-1em]{$B$};
        \draw [decorate,decoration={brace,mirror,amplitude=5pt}]
        (-.7,-.6) -- (.7,-.6) node[midway,xshift=-.3em,yshift=-0.8em]{\small $A$};
        \node at (1.2,-1) {\tiny $A^c=\{s,t\}$};
    \end{tikzpicture}
    =
    \frac{1}{12}
    \begin{tikzpicture}
        \node[ext] (v) at (0,0) {};
        \node at (-.62, .1) {$\scriptscriptstyle \vdots$};
        \node[int] (w) at (.8, 0) {};
        \node at (.7,-.25) {\tiny$p$};
        \node at (.4,-1) {\small$\omega_{B\sqcup p}\otimes\delta_{irr}$};
        \draw (v) edge[->-] +(-.6,-.5) edge[->-] +(-.6,-.3) edge[->-] +(-.3,-.5)
        edge[->-] +(-.6,.5) edge[->-] +(-.6,.3) edge[->-] +(-.3,.5) 
        edge[dashed] +(.6,.5) edge[dashed] +(.6,+.3)  edge[dashed] +(.3,+.5)
        edge[dashed] +(.6,-.5) edge[dashed] +(.6,-.3) edge[dashed] +(.3,-.5) edge[->-] (w);
        \draw (w) edge[crossed, out=35,in=-35,distance=1cm] (w);
        \draw [decorate,decoration={brace,amplitude=5pt,mirror}]
        (-.8,.5) -- (-.8,-.5) node[midway,xshift=-1em]{$B$};
    \end{tikzpicture}
\]

With condition $1)$ and relations $3a),3b), 4)$ we can consider only the generators whose graphical depiction has no loops except possibly at $\bar{v}$, but including the special case of exactly one loop at $\tilde{v}$ for the case $(B)$ with $\omega_B\otimes \delta_{irr}$. In addition, we can ignore case $(A)$ generators with $|A^c|=2$ and a multiple edge parallel to the crossed edge with an endpoint in the set $B$ using coinvariants and the weight $13$ relation:
\[
    2 \cdot
    \begin{tikzpicture}
        \node[ext] (v) at (0,0) {};
        \node at (-.62, .1) {$\scriptscriptstyle \vdots$};
        \node[int] (w) at (.8, 0) {};
        \node (q) at (1.3,.2) {\tiny$q$};
        \node (s) at (.65,.5) {\tiny$s$};
        \node (t) at (.15,.5) {\tiny$t$};
        \draw (v) edge[->-] +(-.6,-.5) edge[->-] +(-.6,-.3) edge[->-] +(-.3,-.5)
        edge[->-] +(-.6,.5) edge[->-] +(-.6,.3) edge[->-] +(-.3,.5) 
        edge[dashed] +(.6,-.5) edge[dashed] +(.6,-.3) edge[dashed] +(.3,-.5) edge[crossed] (w) edge[->-, bend left=60,distance=.5cm] (w);
        \draw (w) edge +(.5,+.5);
        \draw [decorate,decoration={brace,amplitude=5pt,mirror}]
        (-.8,.5) -- (-.8,-.5) node[midway,xshift=-1em]{$B$};
        \draw [decorate,decoration={brace,mirror,amplitude=5pt}]
        (-.7,-.6) -- (.7,-.6) node[midway,yshift=-0.8em]{\small $A$};
    \end{tikzpicture}
    =
    \begin{tikzpicture}
        \node[ext] (v) at (0,0) {};
        \node at (-.62, .1) {$\scriptscriptstyle \vdots$};
        \node[int] (w) at (.8, 0) {};
        \node (q) at (1.3,.2) {\tiny$q$};
        \node (s) at (.65,.5) {\tiny$s$};
        \node (t) at (.15,.5) {\tiny$t$};
        \draw (v) edge[->-] +(-.6,-.5) edge[->-] +(-.6,-.3) edge[->-] +(-.3,-.5)
        edge[->-] +(-.6,.5) edge[->-] +(-.6,.3) edge[->-] +(-.3,.5) 
        edge[dashed] +(.6,-.5) edge[dashed] +(.6,-.3) edge[dashed] +(.3,-.5) edge[crossed] (w) edge[->-, bend left=60,distance=.5cm] (w);
        \draw (w) edge +(.5,+.5);
        \draw [decorate,decoration={brace,amplitude=5pt,mirror}]
        (-.8,.5) -- (-.8,-.5) node[midway,xshift=-1em]{$B$};
        \draw [decorate,decoration={brace,mirror,amplitude=5pt}]
        (-.7,-.6) -- (.7,-.6) node[midway,yshift=-0.8em]{\small $A$};
    \end{tikzpicture}
    +
    \begin{tikzpicture}
        \node[ext] (v) at (0,0) {};
        \node at (-.62, .1) {$\scriptscriptstyle \vdots$};
        \node[int] (w) at (.8, 0) {};
        \node (q) at (1.3,.2) {\tiny$q$};
        \node (s) at (.65,.5) {\tiny$t$};
        \node (t) at (.15,.5) {\tiny$s$};
        \draw (v) edge[->-] +(-.6,-.5) edge[->-] +(-.6,-.3) edge[->-] +(-.3,-.5)
        edge[->-] +(-.6,.5) edge[->-] +(-.6,.3) edge[->-] +(-.3,.5) 
        edge[dashed] +(.6,-.5) edge[dashed] +(.6,-.3) edge[dashed] +(.3,-.5) edge[crossed] (w) edge[->-, bend left=60,distance=.5cm] (w);
        \draw (w) edge +(.5,+.5);
        \draw [decorate,decoration={brace,amplitude=5pt,mirror}]
        (-.8,.5) -- (-.8,-.5) node[midway,xshift=-1em]{$B$};
        \draw [decorate,decoration={brace,mirror,amplitude=5pt}]
        (-.7,-.6) -- (.7,-.6) node[midway,yshift=-0.8em]{\small $A$};
    \end{tikzpicture}
    =
    \begin{tikzpicture}
        \node[ext] (v) at (0,0) {};
        \node at (-.62, .1) {$\scriptscriptstyle \vdots$};
        \node[int] (w) at (.8, 0) {};
        \node (q) at (.2,.5) {\tiny$q$};
        \node (s) at (1.2,.4) {\tiny$s$};
        \node (t) at (1.2,-.4) {\tiny$t$};
        \draw (v) edge[->-] +(-.6,-.5) edge[->-] +(-.6,-.3) edge[->-] +(-.3,-.5)
        edge[->-] +(-.6,.5) edge[->-] +(-.6,.3) edge[->-] +(-.3,.5) 
        edge[->-] +(.6,.5)
        edge[dashed] +(.6,-.5) edge[dashed] +(.6,-.3) edge[dashed] +(.3,-.5) edge[crossed] (w);
        \draw (w)  to [out=35,in=-35,distance=1cm] (w);
        \draw [decorate,decoration={brace,amplitude=5pt,mirror}]
        (-.8,.5) -- (-.8,-.5) node[midway,xshift=-1em]{$B$};
        \draw [decorate,decoration={brace,mirror,amplitude=5pt}]
        (-.7,-.6) -- (.7,-.6) node[midway,xshift=-.3em,yshift=-0.8em]{\small $A$};
    \end{tikzpicture}.
\]

\subsection{Blown-up representation of the generators} \label{sec:BlownUpRep}

We wish to represent the decorated graphs generating $\myB_{g,n}$ in terms of purely combinatorial data, so by graphs with features on their edges and vertices.

Using the graphical depictions of cohomology classes, we look at their connected components after deleting the special vertex but keeping it's half-edges and hairs; this is called the blown-up representation. We keep track of the features at the special vertex by labeling the hairs of the blown-up components $j$, $\omega$ or $\epsilon$, depending upon wether they were a hair of the original decorated graph, an edge in the set $B$ of the decoration, or an edge outside $B$ respectively. In each case there are precisely eleven $\omega$ labels. In the pictures below, one has to imagine  the rest of the ambient graph to exist unchanged, potentially grouping the $\epsilon$ and $\omega$ hairs into connected components.

\[
    \begin{tikzpicture}
        \node[ext] (v) at (0,0) {};
        \node at (-.62, .1) {$\scriptscriptstyle \vdots$};
        \draw (v) edge[->-] +(-.6,-.5) edge[->-] +(-.6,-.3) edge[->-] +(-.3,-.5)
        edge[->-] +(-.6,.5) edge[->-] +(-.6,.3) edge[->-] +(-.3,.5) 
        edge[dashed] +(.6,.5) edge[dashed] +(.6,+.3)  edge[dashed] +(.3,+.5)
        edge[dashed] +(.6,-.5) edge[dashed] +(.6,-.3) edge[dashed] +(.3,-.5) ;
        \draw [decorate,decoration={brace,amplitude=5pt,mirror}]
      (-.8,.5) -- (-.8,-.5) node[midway,xshift=-1em]{$B$};
        \draw [decorate,decoration={brace,amplitude=5pt}]
        (.7,.5) -- (.7,-.5) node[midway,xshift=1em]{$B^c$};
    \end{tikzpicture}
    \mapsto
    \begin{tikzpicture}
        \node (e1) at (-1.5,-.5) {$\epsilon$};
        \node at (-1,-.5) {$\dots$};
        \node (e2) at (-.5,-.5) {$\epsilon$};
        \node (w1) at (0,-.5) {$\omega$};
        \node at (0.5,-.5) {$\dots$};
        \node (w2) at (1,-.5) {$\omega$};
        \draw (e1) edge[dashed] +(0,1) (e2) edge[dashed] +(0,1);
        \draw (w1) edge +(0,1) (w2) edge +(0,1);
        \draw [decorate,decoration={brace,mirror,amplitude=5pt}]
        (-1.5,-.7) -- (-.3,-.7) node[midway,yshift=-0.8em]{\tiny $B^c$};
        \draw [decorate,decoration={brace,mirror,amplitude=5pt}]
        (-.2,-.7) -- (1.2,-.7) node[midway,yshift=-0.8em]{\tiny $B$};
    \end{tikzpicture}
    \hspace{0.6cm}
    \begin{tikzpicture}
        \node[ext] (v) at (0,0) {};
        \node at (-.62, .1) {$\scriptscriptstyle \vdots$};
        \node[int] (w) at (.8, 0) {};
        \draw (v) edge[->-] +(-.6,-.5) edge[->-] +(-.6,-.3) edge[->-] +(-.3,-.5)
        edge[->-] +(-.6,.5) edge[->-] +(-.6,.3) edge[->-] +(-.3,.5) 
        edge[dashed] +(.6,.5) edge[dashed] +(.6,+.3)  edge[dashed] +(.3,+.5)
        edge[dashed] +(.6,-.5) edge[dashed] +(.6,-.3) edge[dashed] +(.3,-.5) edge[crossed] (w);
        \draw (w) edge +(.6,.5) edge[dashed] +(.6,+.3)  edge[dashed] +(.3,+.5)
        edge +(.6,-.5) edge[dashed] +(.6,-.3) edge[dashed] +(.3,-.5);
        \draw [decorate,decoration={brace,amplitude=5pt,mirror}]
        (-.8,.5) -- (-.8,-.5) node[midway,xshift=-1em]{$B$};
        \draw [decorate,decoration={brace,amplitude=5pt}]
        (1.5,.5) -- (1.5,-.5) node[midway,xshift=1em]{$A^c$};
        \draw [decorate,decoration={brace,mirror,amplitude=5pt}]
        (-.7,-.6) -- (.7,-.6) node[midway,yshift=-0.8em]{\small $A$};
    \end{tikzpicture}
    \mapsto
    \begin{tikzpicture}
        \node (e1) at (-1.5,-.5) {$\epsilon$};
        \node at (-1,-.5) {$\dots$};
        \node (e2) at (-.5,-.5) {$\epsilon$};
        \node (w1) at (0,-.5) {$\omega$};
        \node at (0.5,-.5) {$\dots$};
        \node (w2) at (1,-.5) {$\omega$};
        \node (w_) at (1.5,-.5) {$\omega$};
        \node[int] (v) at (1.5,.5) {};
        \draw (e1) edge[dashed] +(0,1) (e2) edge[dashed] +(0,1);
        \draw (w1) edge +(0,1) (w2) edge +(0,1);
        \draw (w_) edge[crossed] (v);
        \draw (v) edge +(.5,.5)  edge[dashed] +(.5,.2) edge[dashed] +(.2,.5);
        \draw (v) edge +(-.5,.5) edge[dashed] +(-.5,.2) edge[dashed] +(-.2,.5);
        \draw [decorate,decoration={brace,mirror,amplitude=5pt}]
        (-1.5,-.7) -- (-.3,-.7) node[midway,yshift=-0.8em]{\tiny $A\setminus B$};
        \draw [decorate,decoration={brace,mirror,amplitude=5pt}]
        (-.2,-.7) -- (1.2,-.7) node[midway,yshift=-0.8em]{\tiny $B$};
        \draw [decorate,decoration={brace,mirror,amplitude=5pt}]
        (2,.5) -- (2,1.1) node[midway,xshift=1em]{\small $A^c$};
    \end{tikzpicture}
    \hspace{1cm}
\]

As in \cite[Section 3.2]{CLPW2}, we classify blown-up representations by their excess value $$E(g,n) = 3g+2n,$$ which fits together as follows with the parameters of blown-up components:
\[ \sum_i\omega_i = 11  \hspace{1cm}  n_{\bar{v}} = \sum_i\epsilon_i+\omega_i \hspace{1cm} n = \sum_i n_i \hspace{1cm} g = 1+\sum_i(g_i+\epsilon_i+\omega_i-1) \]
\[ 2(g-1)+n = n_{\bar{v}}+\sum_i 2(g_i-1)+\epsilon_i+\omega_i+n_i \hspace{1.5cm}  \forall i: 3(g_i-1)+3\epsilon_i+\omega_i+2n_i \geq 0\]
\[ 3(g-1)+2n = n_{\bar{v}}+\sum_i 3(g_i-1)+2(\epsilon_i+\omega_i+n_i) = 22+\sum_i 3(g_i-1)+3\epsilon_i+\omega_i+2n_i,
\]
where $g_i,\epsilon_i,\omega_i,n_i$ are the genus, number of $\epsilon$ labels, number of $\omega$ labels and number of original hairs of the $i$-th blown-up component. $E(g,n)-25$ is additive over paramenter $e(C_i)=3(g_i-1)+3\epsilon_i+\omega_i+2n_i$ of blown-up components, which is non-negative for all non-zero graphs that we have to consider in weight $13$.

One reason for considering this function as a measure of complexity is that the only blown-up components with $e(C_i)=0$ are the double-hair $\omega - j$  and the 'tripod' with $\omega$ labels (see \eqref{equ:VirtualExample}), which are arguably the simplest stable subgraphs.

We call the \emph{virtual} blown-up representation of a decorated graph its list of blown-up components with $e(C_i)>0$; this list determines uniquely the excess value $E(g,n)$. Each virtual blown-up graph can be completed to an actual blown-up representation in different ways by appending $\omega-j$ hairs and tripods. The possible ways in which this can be done determines the range of existence of the virtual blown-up graph among the different $(g,n)$ pairs withing the excess class.

For example, the following virtual blown-up representation of excess $28$ can be completed in three ways.
\begin{multline} \label{equ:VirtualExample}
    \begin{tikzpicture}
        \node (w2) at (.4,-.5) {\small$\omega$};
        \node (w3) at (.8,-.5) {\small$\omega$};
        \node (w4) at (1.2,-.5) {\small$\omega$};
        \node[int] (v1) at (.4,.3) {};
        \node[int] (v2) at (1,.3) {};
        \node (j1) at (.4,1.2) {\small $j$};
        \draw(v1) edge(v2)edge[crossed](w2) edge(j1) (v2) edge(w3)edge(w4);
        \node (w5) at (1.6,-.5) {\small$\omega$};
        \node (w6) at (2,-.5) {\small$\omega$};
        \node[int] (v5) at (1.8,.3) {};
        \node (j) at (1.8,1.2) {\small $j$};
        \draw (v5) edge(j)edge(w5)edge(w6);
    \end{tikzpicture}
    \hspace{.2in} \xrightarrow{\substack{\text{possible completions}\\\text{$(g,n): (4,8),(6,5),(8,2)$}}} \hspace{.2in}
    \begin{tikzpicture}
        \node (ww1) at (-2,-.5) {\small $\omega$};
        \node (jj1) at (-2,.5) {\small $j$};
        \node (ww2) at (-1.6,-.5) {\small $\omega$};
        \node (jj2) at (-1.6,.5) {\small $j$};
        \node (ww3) at (-1.2,-.5) {\small $\omega$};
        \node (jj3) at (-1.2,.5) {\small $j$};
        \node (ww4) at (-.8,-.5) {\small $\omega$};
        \node (jj4) at (-.8,.5) {\small $j$};
        \node (ww5) at (-.4,-.5) {\small $\omega$};
        \node (jj5) at (-.4,.5) {\small $j$};
        \node (ww6) at (0,-.5) {\small $\omega$};
        \node (jj6) at (0,.5) {\small $j$};
        \draw(jj1)edge(ww1)(jj2)edge(ww2)(jj3)edge(ww3)(jj4)edge(ww4)(jj5)edge(ww5)(jj6)edge(ww6);
        \node (w2) at (.4,-.5) {\small$\omega$};
        \node (w3) at (.8,-.5) {\small$\omega$};
        \node (w4) at (1.2,-.5) {\small$\omega$};
        \node[int] (v1) at (.4,.3) {};
        \node[int] (v2) at (1,.3) {};
        \node (j1) at (.4,1.2) {\small $j$};
        \draw(v1) edge(v2)edge[crossed](w2) edge(j1) (v2) edge(w3)edge(w4);
        \node (w5) at (1.6,-.5) {\small$\omega$};
        \node (w6) at (2,-.5) {\small$\omega$};
        \node[int] (v5) at (1.8,.3) {};
        \node (j) at (1.8,1.2) {\small $j$};
        \draw (v5) edge(j)edge(w5)edge(w6);
    \end{tikzpicture}, \\
    \begin{tikzpicture}
        \node (ww1) at (-2,-.5) {\small $\omega$};
        \node (jj1) at (-2,.5) {\small $j$};
        \node (ww2) at (-1.6,-.5) {\small $\omega$};
        \node (jj2) at (-1.6,.5) {\small $j$};
        \node (ww3) at (-1.2,-.5) {\small $\omega$};
        \node (jj3) at (-1.2,.5) {\small $j$};
        \node (ww4) at (-.8,-.5) {\small $\omega$};
        \node (ww5) at (-.4,-.5) {\small $\omega$};
        \node (ww6) at (0,-.5) {\small $\omega$};
        \node[int] (vv2) at (-.4,.3) {};
        \draw(jj1)edge(ww1)(jj2)edge(ww2)(jj3)edge(ww3)  (vv2) edge(ww4)edge(ww5)edge(ww6);
        \node (w2) at (.4,-.5) {\small$\omega$};
        \node (w3) at (.8,-.5) {\small$\omega$};
        \node (w4) at (1.2,-.5) {\small$\omega$};
        \node[int] (v1) at (.4,.3) {};
        \node[int] (v2) at (1,.3) {};
        \node (j1) at (.4,1.2) {\small $j$};
        \draw(v1) edge(v2)edge[crossed](w2) edge(j1) (v2) edge(w3)edge(w4);
        \node (w5) at (1.6,-.5) {\small$\omega$};
        \node (w6) at (2,-.5) {\small$\omega$};
        \node[int] (v5) at (1.8,.3) {};
        \node (j) at (1.8,1.2) {\small $j$};
        \draw (v5) edge(j)edge(w5)edge(w6);
    \end{tikzpicture}, \hspace{.4in}
    \begin{tikzpicture}
        \node (ww1) at (-2,-.5) {\small $\omega$};
        \node (ww2) at (-1.6,-.5) {\small $\omega$};
        \node (ww3) at (-1.2,-.5) {\small $\omega$};
        \node[int] (vv1) at (-1.6,.3) {};
        \node (ww4) at (-.8,-.5) {\small $\omega$};
        \node (ww5) at (-.4,-.5) {\small $\omega$};
        \node (ww6) at (0,-.5) {\small $\omega$};
        \node[int] (vv2) at (-.4,.3) {};
        \draw(vv1) edge(ww1)edge(ww2)edge(ww3) (vv2) edge(ww4)edge(ww5)edge(ww6);
        \node (w2) at (.4,-.5) {\small$\omega$};
        \node (w3) at (.8,-.5) {\small$\omega$};
        \node (w4) at (1.2,-.5) {\small$\omega$};
        \node[int] (v1) at (.4,.3) {};
        \node[int] (v2) at (1,.3) {};
        \node (j1) at (.4,1.2) {\small $j$};
        \draw(v1) edge(v2)edge[crossed](w2) edge(j1) (v2) edge(w3)edge(w4);
        \node (w5) at (1.6,-.5) {\small$\omega$};
        \node (w6) at (2,-.5) {\small$\omega$};
        \node[int] (v5) at (1.8,.3) {};
        \node (j) at (1.8,1.2) {\small $j$};
        \draw (v5) edge(j)edge(w5)edge(w6);
    \end{tikzpicture}
\end{multline}


\section{Computations} \label{sec:Computations}

\subsection{Previous results in excess smaller or equal to 27} 
These cases have been analyzed in \cite[Section 3]{CLPW2}. In excess smaller or equal to $25$ we have $H^*(\myB_{g,n}) = 0$. In excess 26 and 27 the cohomology is concentrated in the top degree of the complex, except in $(g,n)=(9,0)$ where it is in one degree lower.
\begin{align*}
    H^k(\myB_{2,10}) &=
    \begin{cases}
        V_{1^{10}} & \text{for $k=14$} \\
        0 & \text{otherwise}
    \end{cases}
&
H^k(\myB_{4,7}) &=
    \begin{cases}
        V_{1^{7}} & \text{for $k=17$} \\
        0 & \text{otherwise}
    \end{cases}
\\
H^k(\myB_{6,4}) &=
\begin{cases}
    V_{1^{4}} & \text{for $k=20$} \\
    0 & \text{otherwise}
\end{cases}     
&
H^k(\myB_{8,1}) &=
\begin{cases}
    V_1 & \text{for $k=23$} \\
    0 & \text{otherwise.}
\end{cases}  
\end{align*}

\begin{align*}
  H^k(\myB_{1,12}) &=
  \begin{cases}
      V_{21^{10}} & \text{for $k=13$} \\
      0 & \text{otherwise}
  \end{cases}
  &
  H^k(\myB_{3,9}) &=
  \scalebox{.95}{$\begin{cases}
      V_{1^{9}}\oplus V_{21^7}^{\oplus 2} & \text{for $k=16$} \\
      0 & \text{otherwise}
  \end{cases}$
  }
  \\
  H^k(\myB_{5,6}) &=
  \begin{cases}
    V_{1^{6}}\oplus V_{21^4}^{\oplus 2} & \text{for $k=19$} \!\! \\
    0 & \text{otherwise}
\end{cases}
  &
  H^k(\myB_{7,3}) &=
  \begin{cases}
    V_{1^{3}}\oplus V_{21}^{\oplus 2} & \text{for $k=22$} \\
    0 & \text{otherwise}
\end{cases}
  \\
  H^k(\myB_{9,0}) &=
  \begin{cases}
    \mathbb C & \text{for $k=24$} \\
    0 & \text{otherwise.}
  \end{cases}
\end{align*}



\subsection{Excess 28} We are looking at the $(g,n)$ pairs $(2,11),(4,8),(6,5)$ and $(8,2)$.
At the bottom of this paper we have appended a list of 106 virtual blown-up representations of excess $28$ graphs. Above each there is a unique identifier. The second ID above each graph, next to the mapping arrow, is a choice of a non-zero term in the image of the differential; this will be explained below. $\omega,\epsilon$ and $j$ hairs are colored blue, yellow and green respectively. The crossed edge is colored red.

The values of $n$ above each graph are the ones for which it can be completed to a $(g,n)$ decorated graph with $E(g,n)=28$ by adding $\omega-j$ hairs and tripods; we call this the \emph{existence range} of the virtual blown-up representation.

Each virtual blown-up representation in the list also represents all possible ways of labeling its $j$ hairs by numbers $1,...,n$. The $\ss_n$ action on $\myB_{g,n}$ then restricts to the subspace generated by isomorphism classes of these labelings. Thus, each virtual blown-up representation $\VV$ contributes an $\ss_n-$representation $\left<\VV\right>\subseteq\myB_{g,n}$, which we decompose in terms of specht modules.
Note that the actual Specht module contributions displayed above each graph in the list might not be accurate because of weight $13$ or $11$ relations.


\begin{proposition} \label{thm:GK121} For every $(g,n)$ with $E(g,n)=28$, it holds
    \begin{multline*}
        \myB_{g,n} \hspace{.1in} \cong \hspace{.1in}
        \bigoplus_{\VV \in A_3} \left<\VV\right> \hspace{.05in}
        \oplus \bigg( \bigoplus_{\VV \in A_2} \left<\VV \right> \bigg) \Big/ \{\substack{\text{weight 13}\\\text{relations}}\} \\
        \oplus \bigg( \bigoplus_{\VV \in B_1}  \left<\VV \right> \bigg) \Big/ \{\substack{\text{weight 11}\\\text{relations}}\}
        \hspace{.05in} \oplus \bigg( \bigoplus_{\VV \in B_{irr}} \left<\VV \right> \bigg) \Big/ \{\substack{\text{weight 11}\\\text{relations}}\} ,
    \end{multline*}
    where $A_3,A_2,B_1$ and $B_{irr}$ are the four families of virtual blown-up representations that we have listed at the bottom of this paper. It is understood that a graph from the list only appears in the direct sums if it exists for the specific pair $(g,n)$.
\end{proposition}

The proof of \ref{thm:GK121} relies on the correctness of the algorithm that we have used to generate all virtual blown-up representations, which we will explain in Section \ref{sec:ComputerProgram}. \\

\begin{table}[!h]
    \caption{Amount of virtual blown-up representations in excess 28}
    \label{fig:FamiliesDistribution}
    \begin{tabular}{c|c|c|c|c}
        edges & $A_3$ & $A_2$ & $B_1$ & $B_{irr}$  \\ \hline
        $10-n$ & 6 & & & \\ \hline
        $11-n$ & 17 & & 7 & 4 \\ \hline
        $12-n$ & 12 & 8 & 19 & 9 \\ \hline
        $13-n$ &  & 8 & 12 & 4
    \end{tabular}
\end{table}

In excess $28$ it is possible to resolve the weight $13$ and $11$ relations by ignoring virtual blown-up representations indipendently of the $(g,n)$ pair.

\begin{lemma} For every $(g,n)$ with $E(g,n)=28$, it holds
    \[  \myB_{g,n} \hspace{.1in} \cong \hspace{.1in} \bigoplus_{\VV \in A_3} \left<\VV\right> \hspace{.05in} \oplus \bigoplus_{\VV \in \tilde{A}_2} \left<\VV \right> \oplus \bigoplus_{\VV \in \tilde{B}_1}  \left<\VV \right> \oplus \bigoplus_{\VV \in \tilde{B}_{irr}} \left<\VV \right> , \]
    where $\tilde{A}_2 = A_2 \setminus \{ 71,69,75,72,108,105,111,107 \}$, $\tilde{B}_1 = B_1\setminus \{ 79 \}$ and $\tilde{B}_{irr} = B_{irr} \setminus \{96\}$.
\end{lemma}
\begin{proof}
    The algorithm that we used to generate all virtual blown-up representations also groups them by weight 13 and weight 11 relations. Any graphs that might have relations between them are listed in the same line.
    In our case, all relations restrict to the following six groups of graphs.
    \begin{align*}
       \text{edges 12-n: } & & \{73,70,\underline{71},\underline{69} \} & & \{76,74,\underline{75},\underline{72}\} & & \{77,\underline{79},84\} & &  \{97,\underline{96},100 \} \\
       \text{edges 13-n: } & & \{110,106,\underline{108},\underline{105} \} & & \{112,109,\underline{111},\underline{107}\}
    \end{align*}
    One checks that the ten underlined graphs are redundant, and that removing them makes all other graphs in their relative groups indipendent. Moreover, one also checks that this holds indipendently of $(g,n)$.
\end{proof}


\begin{theorem} \label{thm:reducedGK}
    For each $(g,n)$ with $E(g,n)=28$, the projection
    $$ \myB_{g,n} \hspace{.1in}\twoheadrightarrow \hspace{.1in} \left<\VV_{A_2}^{\epsilon,j}\right> \oplus  \left<\VV_{irr}^{\omega-\epsilon,j}\right> \oplus  \left<\VV_{irr}^{\epsilon,j}\right>
    \oplus  \left<\VV_{A_2}^{i,j}\right> \oplus  \left<\VV_{irr}^{i\vee j}\right> \oplus \left<\VV_{irr}^{i\wedge j}\right> $$
    is a quasi-isomorphism, where the six virtual blown-up representations have IDs 70, 98, 97, 106, 125 and 126 respectively (see Table \ref{fig:MatrixDiff4}). It is understood that a graph only appears in the direct sum if it exists for the specific $(g,n)$ pair.
\end{theorem}
\begin{proof}
    We have to compute the images of the three differentials between the four degree classes. Even though the differential acts on individual decorated graphs, since it respects the labeling $1,...,n$ of the hairs we can restrict to describe its action on virtual blown-up representations.
    Thus, we can give matrices of the respective differentials with rows and columns indexed by virtual blown-up representation from our list, they are displayed in Tables \ref{fig:MatrixDiff1}, \ref{fig:MatrixDiff2} and \ref{fig:MatrixDiff3}. The graphs in the rows and columns are in order as they appear in our list.
    
    A white or black cell in each matrix means that the graph in the column has the graph in the row as a zero or non-zero term in its differential respectively. A gray cell means that the value will not be relevant to our computation, just as every other graph in the codomain that doesn't appear amongst the rows of the matrix. \\
    
    Black cells also represent choices of non-zero leading terms in the differentials of each graph. They can be used to carry out gaussian elimination, as the rows can be permuted into a lower triangular form. This holds for all rows except 116 and 117 in Table \ref{fig:MatrixDiff3}. Nonetheless, one checks that the images of graphs 80 and 81 are indipendent and thus span the whole subspace in edge group $13-n$ generated by graphs 116 and 117.

    We emphasize that the elimination process is indipendent of $(g,n)$. Our choices of leading terms preserve the existence range of each virtual blown-up representation. Thus, for each $(g,n)$ pair, either both graphs cancel out in cohomology or they both don't exist. Moreover, these choices also preserve the $\ss_n-$representation contributed to $\myB_{g,n}$.
    
    The relations induced by the image of edge group $e$ on the edge group $e+1$ allow us to ignore the chosen leading terms when computing the image of edge group $e+1$. This is why the columns of each matrix only contain the graphs which don't appear amongst the rows with a black cell of the previous matrix. These leading terms are the second ID displayed above each graph in our list, where the mapping arrow indicates which is the argument and which is the image.

    As a practical note, it is important to exploit gaussian elimination as much as possible by choosing to kill the graphs with the most complicated differential, i.e. $A_3$ graphs and the ones with $\epsilon$ hairs. \\

    We can conclude that annihilating all virtual blown-up representations appearing in a row or column with a black cell of one of the matrices defines a quasi-isomorphism.
    
    One checks that the only (non-redundant) virtual blown-up representations remaining from the whole list are the six ones stated in the theorem. Non-zero terms in the differential of the three remaining graphs in edge group $12-n$ are colored orange. But these terms do not fully annihilate the whole $\ss_n-$representation contributed by the graphs in the rows, so we cannot eliminate them.
\end{proof}

\begin{table}[!h]
    \caption{Differential between degrees $23-n$ and $24-n$}
    \label{fig:MatrixDiff1}
    \vspace{-.1in} \small
    \begin{tabular}{ c c|c|c|c|c|c|c }
        & $10-n$ & \multicolumn{6}{c}{ $A_3$ } \\
        $11-n$ & & 23 & 24 & 25 & 26 & 27 & 28 \\ \hline
        \multirow{5}{*}{ $A_3$ }
        & 34 & \cellcolor{lightgray} & \cellcolor{darkgray} & & & & \\ \cline{2-8}
        & 39 & \cellcolor{lightgray} & \cellcolor{lightgray} & \cellcolor{darkgray} & & & \\ \cline{2-8}
        & 41 & \cellcolor{lightgray} & \cellcolor{lightgray} & \cellcolor{lightgray} & \cellcolor{darkgray} & & \\ \cline{2-8}
        & 44 & \cellcolor{lightgray} & \cellcolor{lightgray} & \cellcolor{lightgray} & \cellcolor{lightgray} & \cellcolor{darkgray} & \\ \cline{2-8}
        & 45 & \cellcolor{lightgray} & \cellcolor{lightgray} & \cellcolor{lightgray} & \cellcolor{lightgray} & \cellcolor{lightgray} & \cellcolor{darkgray} \\ \hline
        \multirow{1}{*}{ $B_1$ }
        & 46 & \cellcolor{darkgray} & & & & & \\
    \end{tabular}
    \normalsize
\end{table}

\begin{table}[!h]
    \caption{Differential between degrees $24-n$ and $25-n$}
    \label{fig:MatrixDiff2}
    \vspace{-.1in} \hspace{-.3in} \tiny
    \begin{tabular}{ c c|c|c|c|c|c|c|c|c|c|c|c|c|c|c|c|c|c|c|c|c|c|c }
        & $11-n$& \multicolumn{12}{c|}{ $A_3$ } & \multicolumn{6}{c|}{ $B_1$ } & \multicolumn{4}{c}{ $B_{irr}$ } \\
        $12-n$ & & 29 & 30 & 31 & 32 & 33 & 35 & 36 & 37 & 38 & 40 & 42 & 43 &   47 & 48 & 49 & 50 & 52 & 51 & 53 & 54 & 55 & 56 \\ \hline
        \multirow{7}{*}{ $A_3$ }
        & 61 & \cellcolor{lightgray} & \cellcolor{lightgray} & \cellcolor{lightgray} & \cellcolor{darkgray} & & & & & & & & & & & & & & & & & & \\ \cline{2-14}
        & 62 & \cellcolor{lightgray} & \cellcolor{lightgray} & \cellcolor{lightgray} & \cellcolor{lightgray} & \cellcolor{lightgray} & \cellcolor{darkgray} & & & & & & & & & & & & & & & \\ \cline{2-14}
        & 64 & \cellcolor{lightgray} & \cellcolor{lightgray} & \cellcolor{lightgray} & \cellcolor{lightgray} & \cellcolor{lightgray} & \cellcolor{lightgray} & \cellcolor{darkgray} & & & & & & & & & & & & & & \\ \cline{2-14}
        & 65 & \cellcolor{lightgray} & \cellcolor{lightgray} & \cellcolor{lightgray} & \cellcolor{lightgray} & \cellcolor{lightgray} & \cellcolor{lightgray} & \cellcolor{lightgray} & \cellcolor{darkgray} & \cellcolor{lightgray} & & & & & & & & & & & & & \\ \cline{2-14}
        & 66 & \cellcolor{lightgray} & \cellcolor{lightgray} & \cellcolor{lightgray} & \cellcolor{lightgray} & \cellcolor{lightgray} & \cellcolor{lightgray} & \cellcolor{lightgray} & \cellcolor{lightgray} & \cellcolor{lightgray} & \cellcolor{darkgray} & & & & & & & & & & & & \\ \cline{2-14}
        & 67 & \cellcolor{lightgray} & \cellcolor{lightgray} & \cellcolor{lightgray} & \cellcolor{lightgray} & \cellcolor{lightgray} & \cellcolor{lightgray} & \cellcolor{lightgray} & \cellcolor{lightgray} & \cellcolor{lightgray} & \cellcolor{lightgray} & \cellcolor{darkgray} & & & & & & & & & & &  \\ \cline{2-14}
        & 68 & \cellcolor{lightgray} & \cellcolor{lightgray} & \cellcolor{lightgray} & \cellcolor{lightgray} & \cellcolor{lightgray} & \cellcolor{lightgray} & \cellcolor{lightgray} & \cellcolor{lightgray} & \cellcolor{lightgray} & \cellcolor{lightgray} & \cellcolor{lightgray} & \cellcolor{darkgray} & & & & & & & & & & \\ \cline{1-14}
        \multirow{1}{*}{ $A_2$ }
        & 73 & \cellcolor{lightgray} & \cellcolor{darkgray} & & & & & & & & & & & & & & & & & & & & \\ \cline{1-14}
        \multirow{10}{*}{ $B_1$ }
        & 78 & \cellcolor{darkgray} & & & & & & & & & & & & & & & & & & & & & \\ \cline{2-20}
        & 82 & \cellcolor{lightgray} & \cellcolor{lightgray} & \cellcolor{lightgray} & \cellcolor{lightgray} & \cellcolor{darkgray} & & & & & & & & & & & & & & & & & \\ \cline{2-20}
        & 86 & \cellcolor{lightgray} & \cellcolor{lightgray} & \cellcolor{lightgray} & \cellcolor{lightgray} & \cellcolor{lightgray} & \cellcolor{lightgray} & \cellcolor{lightgray} & \cellcolor{lightgray} & \cellcolor{lightgray} & \cellcolor{lightgray} & \cellcolor{lightgray} & \cellcolor{lightgray} & \cellcolor{darkgray} & & & & & & & & & \\ \cline{2-20}
        & 88 & \cellcolor{lightgray} & \cellcolor{lightgray} & \cellcolor{lightgray} & \cellcolor{lightgray} & \cellcolor{lightgray} & \cellcolor{lightgray} & \cellcolor{lightgray} & \cellcolor{lightgray} & \cellcolor{lightgray} & \cellcolor{lightgray} & \cellcolor{lightgray} & \cellcolor{lightgray} & \cellcolor{lightgray} & \cellcolor{lightgray} & \cellcolor{darkgray} & & & & & & & \\ \cline{2-20}
        & 89 & \cellcolor{lightgray} & \cellcolor{lightgray} & \cellcolor{lightgray} & \cellcolor{lightgray} & \cellcolor{lightgray} & \cellcolor{lightgray} & \cellcolor{lightgray} & \cellcolor{lightgray} & \cellcolor{lightgray} & \cellcolor{lightgray} & \cellcolor{lightgray} & \cellcolor{lightgray} & \cellcolor{lightgray} & \cellcolor{darkgray} & & & & & & & & \\ \cline{2-20}
        & 91 & \cellcolor{lightgray} & \cellcolor{lightgray} & \cellcolor{lightgray} & \cellcolor{lightgray} & \cellcolor{lightgray} & \cellcolor{lightgray} & \cellcolor{lightgray} & \cellcolor{lightgray} & \cellcolor{lightgray} & \cellcolor{lightgray} & \cellcolor{lightgray} & \cellcolor{lightgray} & \cellcolor{lightgray} & \cellcolor{lightgray} & \cellcolor{lightgray} & \cellcolor{darkgray} & & & & & & \\ \cline{2-20}
        & 94 & \cellcolor{lightgray} & \cellcolor{lightgray} & \cellcolor{lightgray} & \cellcolor{lightgray} & \cellcolor{lightgray} & \cellcolor{lightgray} & \cellcolor{lightgray} & \cellcolor{lightgray} & \cellcolor{lightgray} & \cellcolor{lightgray} & \cellcolor{lightgray} & \cellcolor{lightgray} & \cellcolor{lightgray} & \cellcolor{lightgray} & \cellcolor{lightgray} & \cellcolor{lightgray} & \cellcolor{darkgray} & & & & & \\ \cline{2-20}
        & 93 & \cellcolor{lightgray} & \cellcolor{lightgray} & \cellcolor{lightgray} & \cellcolor{lightgray} & \cellcolor{lightgray} & \cellcolor{lightgray} & \cellcolor{lightgray} & \cellcolor{lightgray} & \cellcolor{lightgray} & \cellcolor{lightgray} & \cellcolor{lightgray} & \cellcolor{lightgray} & \cellcolor{lightgray} & \cellcolor{lightgray} & \cellcolor{lightgray} & \cellcolor{lightgray} & \cellcolor{lightgray} & \cellcolor{darkgray} & & & \\ \cline{2-20}
        & 77 & \cellcolor{lightgray} & \cellcolor{lightgray} & \cellcolor{darkgray} & & & & & & & & & & & & & & & & & & & \\ \cline{2-20}
        & 84 & \cellcolor{lightgray} & \cellcolor{lightgray} & \cellcolor{lightgray} & \cellcolor{lightgray} & \cellcolor{lightgray} & \cellcolor{lightgray} & \cellcolor{lightgray} & & \cellcolor{darkgray} & & & & & & & & & & & & & \\ \cline{1-24}
        \multirow{4}{*}{ $B_{irr}$ }
        & 101 & \cellcolor{lightgray} & \cellcolor{lightgray} & \cellcolor{lightgray} & \cellcolor{lightgray} & \cellcolor{lightgray} & \cellcolor{lightgray} & \cellcolor{lightgray} & \cellcolor{lightgray} & \cellcolor{lightgray} & \cellcolor{lightgray} & \cellcolor{lightgray} & \cellcolor{lightgray} & & & & & & & \cellcolor{lightgray} & \cellcolor{darkgray} & \\ \cline{2-24}
        & 103 & \cellcolor{lightgray} & \cellcolor{lightgray} & \cellcolor{lightgray} & \cellcolor{lightgray} & \cellcolor{lightgray} & \cellcolor{lightgray} & \cellcolor{lightgray} & \cellcolor{lightgray} & \cellcolor{lightgray} & \cellcolor{lightgray} & \cellcolor{lightgray} & \cellcolor{lightgray} & & & & & & & \cellcolor{lightgray} & \cellcolor{lightgray} & \cellcolor{darkgray} \\ \cline{2-24}
        & 104 & \cellcolor{lightgray} & \cellcolor{lightgray} & \cellcolor{lightgray} & \cellcolor{lightgray} & \cellcolor{lightgray} & \cellcolor{lightgray} & \cellcolor{lightgray} & \cellcolor{lightgray} & \cellcolor{lightgray} & \cellcolor{lightgray} & \cellcolor{lightgray} & \cellcolor{lightgray} & & & & & & & \cellcolor{lightgray} & \cellcolor{lightgray} & \cellcolor{lightgray} & \cellcolor{darkgray} \\ \cline{2-24}
        & 100 & \cellcolor{lightgray} & \cellcolor{lightgray} & \cellcolor{lightgray} & \cellcolor{lightgray} & \cellcolor{lightgray} & \cellcolor{lightgray} & \cellcolor{lightgray} & \cellcolor{lightgray} & \cellcolor{lightgray} & \cellcolor{lightgray} & \cellcolor{lightgray} & \cellcolor{lightgray} & & & & & & & \cellcolor{darkgray} & & &
    \end{tabular}
    \normalsize \vspace{.2in}
\end{table}

\begin{table}[!h]
    \caption{Differential between degrees $25-n$ and $26-n$}
    \label{fig:MatrixDiff3}
    \vspace{-.1in} \tiny
    \begin{tabular}{ c c|c|c|c|c|c|c|c|c|c|c|c|c|c|c|c|c|c|c|c|c }
        & $12-n$ & \multicolumn{5}{c|}{ $A_3$ } & \multicolumn{3}{c|}{$A_2$} & \multicolumn{8}{c|}{$B_1$} & \multicolumn{4}{c}{$B_{irr}$}\\
        $13-n$ & & 57 & 58 & 59 & 60 & 63 &   70 & 76 & 74   & 80 & 81 & 87 & 85 & 83 & 90 & 92 & 95 & 98 &   99 & 102 & 97 \\ \hline
        \multirow{4}{*}{$A_2$}
        & 110 & \cellcolor{lightgray} & \cellcolor{lightgray} & \cellcolor{lightgray} & \cellcolor{darkgray} & & & & & & & & & & & & & & & \\ \cline{2-10}
        & 106 & \cellcolor{lightgray} & \cellcolor{lightgray} & \cellcolor{lightgray} & \cellcolor{lightgray} & \cellcolor{lightgray} & \cellcolor{orange} & & & & & & & & & & & & & & \\ \cline{2-10}
        & 112 & \cellcolor{lightgray} & \cellcolor{lightgray} & \cellcolor{lightgray} & \cellcolor{lightgray} & \cellcolor{lightgray} & & \cellcolor{darkgray} & & & & & & & & & & & & \\ \cline{2-10}
        & 109 & \cellcolor{lightgray} & \cellcolor{lightgray} & \cellcolor{lightgray} & \cellcolor{lightgray} & \cellcolor{lightgray} & \cellcolor{orange} & & \cellcolor{darkgray} & & & & & & & & & & & \\ \cline{1-10}
        \multirow{12}{*}{$B_1$}
        & 113 & \cellcolor{darkgray} & & & & & & & & & & & & & & & & & & \\ \cline{2-18}
        & 115 & \cellcolor{lightgray} & \cellcolor{darkgray} & & & & & & & & & & & & & & & & & \\ \cline{2-18}
        & 114 & \cellcolor{lightgray} & \cellcolor{lightgray} & \cellcolor{darkgray} & & & & & & & & & & & & & & & & \\ \cline{2-18}
        & 116 & \cellcolor{lightgray} & \cellcolor{lightgray} & \cellcolor{lightgray} & \cellcolor{lightgray} & \cellcolor{lightgray} & & & & \cellcolor{darkgray} & \cellcolor{lightgray} & & & & & & & & & \\ \cline{2-18}
        & 117 & \cellcolor{lightgray} & \cellcolor{lightgray} & \cellcolor{lightgray} & \cellcolor{lightgray} & \cellcolor{lightgray} & & & & \cellcolor{lightgray} & \cellcolor{darkgray} & & & & & & & & & \\ \cline{2-18}
        & 121 & \cellcolor{lightgray} & \cellcolor{lightgray} & \cellcolor{lightgray} & \cellcolor{lightgray} & \cellcolor{lightgray} & & & & & & \cellcolor{darkgray} & & & & & & & & \\ \cline{2-18}
        & 118 & \cellcolor{lightgray} & \cellcolor{lightgray} & \cellcolor{lightgray} & \cellcolor{lightgray} & \cellcolor{darkgray} & & & & & & & & & & & & & & \\ \cline{2-18}
        & 120 & \cellcolor{lightgray} & \cellcolor{lightgray} & \cellcolor{lightgray} & \cellcolor{lightgray} & \cellcolor{lightgray} & & & & & & & \cellcolor{darkgray} & & & & & & & \\ \cline{2-18}
        & 119 & \cellcolor{lightgray} & \cellcolor{lightgray} & \cellcolor{lightgray} & \cellcolor{lightgray} & \cellcolor{lightgray} & & & & & & & & \cellcolor{darkgray} & & & & & & \\ \cline{2-18}
        & 122 & \cellcolor{lightgray} & \cellcolor{lightgray} & \cellcolor{lightgray} & \cellcolor{lightgray} & \cellcolor{lightgray} & & & & & & & & & \cellcolor{darkgray} & & & & & \\ \cline{2-18}
        & 123 & \cellcolor{lightgray} & \cellcolor{lightgray} & \cellcolor{lightgray} & \cellcolor{lightgray} & \cellcolor{lightgray} & & & & & & & & & & \cellcolor{darkgray} & & & & \\ \cline{2-18}
        & 124 & \cellcolor{lightgray} & \cellcolor{lightgray} & \cellcolor{lightgray} & \cellcolor{lightgray} & \cellcolor{lightgray} & & & & & & & & & & & \cellcolor{darkgray} & & & \\ \hline
        \multirow{4}{*}{$B_{irr}$}
        & 125 & \cellcolor{lightgray} & \cellcolor{lightgray} & \cellcolor{lightgray} & \cellcolor{lightgray} & \cellcolor{lightgray} & & & & & & & & & & & & & & & \cellcolor{orange} \\ \cline{2-22}
        & 126 & \cellcolor{lightgray} & \cellcolor{lightgray} & \cellcolor{lightgray} & \cellcolor{lightgray} & \cellcolor{lightgray} & & & & & & & & & & & & \cellcolor{orange} & & & \\ \cline{2-22}
        & 127 & \cellcolor{lightgray} & \cellcolor{lightgray} & \cellcolor{lightgray} & \cellcolor{lightgray} & \cellcolor{lightgray} & & & & & & & & & & & & \cellcolor{orange} & \cellcolor{darkgray} & & \cellcolor{orange} \\ \cline{2-22}
        & 128 & \cellcolor{lightgray} & \cellcolor{lightgray} & \cellcolor{lightgray} & \cellcolor{lightgray} & \cellcolor{lightgray} & & & & & & & & & & & & \cellcolor{orange} & & \cellcolor{darkgray} &
    \end{tabular}
    \normalsize
\end{table}

\newpage

\begin{corollary} \label{thm:GKCOHOMOLOGY}
    The cohomology groups of $\myB_{g,n}$ are as follows.
    \begin{align*}
        H^k(\myB_{2,11}) &=
        \scalebox{.95}{$\begin{cases}
            V_{1^{11}} & \text{for $k=14$} \\
            V_{21^{9}} \oplus V_{221^{7}} \oplus V_{31^{8}}^{\oplus 2} & \text{for $k=15$} \\
            0 & \text{otherwise}
        \end{cases}$}
        &
        H^k(\myB_{4,8}) &=
        \scalebox{.95}{$\begin{cases}
            V_{1^{8}}^{\oplus 2} & \text{for $k=17$} \\
            V_{21^{6}} \oplus V_{221^{4}} \oplus V_{31^{5}}^{\oplus 3} & \text{for $k=18$} \\
            0 & \text{otherwise}
        \end{cases}$}
        \\
        H^k(\myB_{6,5}) &=
        \scalebox{.95}{$\begin{cases}
            V_{1^{5}}^{\oplus 2} & \text{for $k=20$} \\
            V_{21^{3}} \oplus V_{221} \oplus V_{31^{2}}^{\oplus 3} & \text{for $k=21$} \\
            0 & \text{otherwise}
        \end{cases}$}
        &
        H^k(\myB_{8,2}) &=
        \scalebox{.95}{$\begin{cases}
            V_{1^{2}}^{\oplus 2} & \text{for $k=23$} \\
            0 & \text{otherwise}
        \end{cases}$}
    \end{align*}
    In particular, they are no longer concentrated in top degree, and as in excess 2, the pair of maximal genus $(8,2)$ vanishes in top degree.
\end{corollary}
\begin{proof}
    It is sufficient to compute the cohomology of the complex on the right-hand side of Theorem \ref{thm:reducedGK}. We display its differential in Table \ref{fig:MatrixDiff4}. This time we have to expand the irreducible $\ss_n-$representations because some in the domain lie in the kernel and some in the codomain only exist for $n\geq 5$. As the differential is $\ss_n-$equivariant, the only non-zero terms con be between specht modules corresponding to the same partition of $n$.
    
    Not all graphs in the image of the differential are amongst the rows (see the orange cells of Table \ref{fig:MatrixDiff3}), but one checks that those have been projected to zero through colmn operations in Theorem \ref{thm:reducedGK}. Thus, the only non-zero terms in the projection are the displayed black cells, whenever they exist.
    The kernel and image can then be directly read off of Table \ref{fig:MatrixDiff4} to compute the cohomology groups in each $(g,n)$.
\end{proof}

\begin{table}[!h]
    \caption{Differential after gaussian elimination}
    \label{fig:MatrixDiff4}
    \vspace{-.1in}
    \begin{tabular}{ c c|c|c|c|c|c|c }
        & $12-n$ edges
        & \multicolumn{2}{c|}{ $\substack{ \VV_{A_2}^{\epsilon,j} \text{\;\;ID: 70} \\ \text{n: 11, 8, 5, 2} \\
            \begin{tikzpicture}
                \node (w1) at (0,-.6) {\tiny$\omega$};
                \node (w2) at (.5,-.6) {\tiny$\epsilon$};
                \node[int] (v1) at (0.25,0) {};
                \node (j2) at (.25,.6) {\tiny$j$};
                \draw (v1) edge[crossed](w1) edge(w2)edge(j2);
            \end{tikzpicture}
            }$ }
        & \multicolumn{2}{c|}{ $\substack{ \VV_{irr}^{\omega-\epsilon,j} \text{\;\;ID: 98} \\ \text{n: 8, 5, 2} \\
            \begin{tikzpicture}
                \node (w1) at (0,-.5) {\tiny$\omega$};
                \node (w2) at (.4,-.5) {\tiny$\epsilon$};
                \draw (w1) edge[bend left=80, distance=.2cm] (w2);
                \node (w3) at (.8,-.6) {\tiny$\omega$};
                \node (w4) at (1.2,-.6) {\tiny$\omega$};
                \node (j2) at (1,.6) {\tiny$j$};
                \node[int] (v2) at (1,0) {};
                \draw(v2)edge(w3)edge(w4)edge(j2);
                \node (ww) at (1.6,-.6) {\tiny$\omega$};
                \node[int] (vv) at (1.6,0) {};
                \draw (vv) edge(ww) edge[loop, crossed, distance=0.4cm] (vv);
            \end{tikzpicture}
            }$ }
        & \multicolumn{2}{c}{ $\substack{ \VV_{irr}^{\epsilon,j} \text{\;\;ID: 97} \\ \text{n: 11, 8, 5, 2} \\
            \begin{tikzpicture}
                \node (w1) at (0,-.6) {\tiny$\epsilon$};
                \node (w2) at (.8,-.6) {\tiny$\omega$};
                \node[int] (v1) at (0.8,0) {};
                \node (j2) at (0,.4) {\tiny$j$};
                \draw (w1) edge(j2) (v1) edge(w2) edge[loop, crossed, distance=0.4cm](v1);
            \end{tikzpicture}
            }$ }
        \\  $13-n$ edges  &
        & $V_{1^{n}}$ & $V_{21^{n-2}}$ & $V_{1^{n}}$ & $V_{21^{n-2}}$ & $V_{1^{n}}$ & $V_{21^{n-2}}$ 
        \\[.1cm] \hline
        \multirow{5}{*}{$\substack{ \VV_{A_2}^{i,j} \text{\;\;ID: 106} \\ \text{n: 11, 8, 5, 2 } \\
            \begin{tikzpicture}
                \node (w1) at (0,-.6) {\tiny$\omega$};
                \node (w2) at (.5,-.6) {\tiny$\omega$};
                \node[int] (v1) at (0,0) {};
                \node[int] (v2) at (.5,0) {};
                \node (j1) at (0,.6) {\tiny$i$};
                \node (j2) at (.5,.6) {\tiny$j$};
                \draw (v1) edge(v2) edge[crossed](w1)edge(j1) (v2)edge(w2)edge(j2);
            \end{tikzpicture}
            }$}
        & $V_{1^{n}}$ & \cellcolor{darkgray} & &  & & & \\ \cline{2-8}
        & $V_{21^{n-2}}$ & & \cellcolor{darkgray} & &  & & \\ \cline{2-8}
        & \tiny$(n\geq 5)$\normalsize $V_{221^{n-4}}$ & & & & & & \\ \cline{2-8}
        & \tiny$(n\geq 5)$\normalsize $V_{21^{n-2}}$ &  & & & & & \\ \cline{2-8}
        & \tiny$(n\geq 5)$\normalsize $V_{31^{n-3}}$ & & & & & & \\ \hline
        \multirow{2}{*}[-1mm]{$\substack{ \VV_{irr}^{i\vee j} \text{\;\;ID: 125} \\ \text{n: 11, 8, 5, 2 } \\
            \begin{tikzpicture}
                \node (w1) at (.25,-.6) {\tiny$\omega$};
                \node (w2) at (.75,-.6) {\tiny$\omega$};
                \node[int] (v1) at (.25,0) {};
                \node[int] (v2) at (.75,0) {};
                \node (j1) at (0,.6) {\tiny$i$};
                \node (j2) at (.5,.6) {\tiny$j$};
                \draw (v1) edge(w1)edge(j1)edge(j2) (v2)edge(w2) edge[loop, crossed, distance=0.4cm] (v2);
            \end{tikzpicture}
            }$}
        & $V_{21^{n-2}}$ & & & & & & \cellcolor{darkgray} \\[.7cm] \cline{2-8}
        & \tiny$(n\geq 5)$\normalsize $V_{31^{n-3}}$ & &  & & & & \\[.7cm] \hline
        \multirow{2}{*}[-1mm]{$\substack{ \VV_{irr}^{i\wedge j} \text{\;\;ID: 126} \\ \text{n: 8, 5, 2 } \\
            \begin{tikzpicture}
                \node (w1) at (0,-.6) {\tiny$\omega$};
                \node (w2) at (.4,-.6) {\tiny$\omega$};
                \node (j1) at (0.2,.6) {\tiny$i$};
                \node[int] (v1) at (0.2,0) {};
                \draw(v1)edge(w1)edge(w2)edge(j1);
                \node (w1) at (.8,-.6) {\tiny$\omega$};
                \node (w2) at (1.2,-.6) {\tiny$\omega$};
                \node (j2) at (1,.6) {\tiny$j$};
                \node[int] (v2) at (1,0) {};
                \draw(v2)edge(w3)edge(w4)edge(j2);
                \node (ww) at (1.5,-.6) {\tiny$\omega$};
                \node[int] (vv) at (1.5,0) {};
                \draw (vv) edge(ww) edge[loop, crossed, distance=0.4cm] (vv);
            \end{tikzpicture}
            }$}
        & $V_{21^{n-2}}$  & & & & \cellcolor{darkgray} & & \\[.7cm] \cline{2-8}
        & \tiny$(n\geq 5)$\normalsize $V_{31^{n-3}}$  & & & & & & \\[.7cm]
    \end{tabular}
\end{table}

\newpage

\section{Computer program} \label{sec:ComputerProgram}

We use Python to generate all possible virtual blown-up representations with a certain upper excess bound $E_{max}$. The script is in the format of a Jupyter Notebook. The graphs are generated using Sage's builtin interface for the Nauty library and are later wrapped in Python objects. We use Pandas to store the large amounts of python objects in dataframes in order to use data management tools like querying, sorting and grouping. Finally, we use Matplotlib to either show the graphs in the output cells or save them in a pdf file. The project is saved in a GitHub repository \cite{github}, where one can also find the virtual blown-up lists for excess 26,27 and 29. We have not thoroughly verified the excess 29 list as there are almost 500 virtual blown-up representations, but a first computation suggests that only 12 may remain after the gaussian elimination process.

We emphasize that computer calculations are hard and allow room for many oversights, so our results should be taken with a grain of salt.

\subsection{Classification of the generators}
From now on we think of generators of $\bGK^{12,1}_{g,n}$ in terms of their graphical depiction, i.e. as stable $(g,n)-$graphs $G$ with additional graph-theoretical data that uniquely determines the cohomology decoration (up to a sign ambiguity due to the orderings).
In particular,
\begin{itemize}
    \item case $(A)$ generators are thought of as having two special vertices $\bar{v}$, $\tilde{v}$ connected by a crossed internal edge.
    \item case $(B)$ generators with $\delta_{irr}$ decoration are thought of as having a loop at $\tilde{v}$ and $g_{\tilde{v}}=0$ instead of $g_{\tilde{v}}=1$.
    \item case $(B)$ generators with $\delta\{\substack{A \\ A'}\}$ decoration have three special vertices $\bar{v},\tilde{v}_1,\tilde{v}_2$, where the last two are connected by a crossed internal edge.
\end{itemize}

There are four families of generators corresponding to the direct sums in \eqref{equ:G}, which we call $A_3,A_2,B_1$, and $B_{irr}$ respectively. Denote by $G_{\bar{v},B}^{\tilde{v}}$, $G_{\bar{v},B}^{\tilde{v}_1,\tilde{v}_2}$, $G_{\bar{v},B}^{\tilde{v},irr}$ the decorated graphs in each family, where the variables range in each case over the following objects:
\begin{enumerate}
    \item[\large $\substack{A_3,A_2\\ B_1,B_{irr}}$\normalsize : ] Any $(g,n)$-stable graph $G$ with exactly one genus $1$ vertex $\bar{v}$ and every other vertex having genus $0$.
    \item[$A_3,A_2$: ] Any choice of a crossed internal edge ($\bar{v},\tilde{v}$) adjacent to $\bar{v}$. We denote by $N_{\bar{v}}, N_{\tilde{v}}$ the set of hairs and half-edges adjacent to $\bar{v}$ or $\tilde{v}$ respectively (excluding the two crossed half-edges). We fix a canonical ordering of the set $N_{\bar{v}}$. In addition, there's any choice of a subset $B\subseteq N_{\bar{v}}$ with $|B|=10$. Stability at $\tilde{v}$ then implies $|N_{\tilde{v}}|\geq 3$ or $|N_{\tilde{v}}|=2$; this condition differentiates the families $A_3$ and $A_2$. $G$ is required to be simple, except possibly with loops at $\bar{v}$ or multiple edges adjacent to $\bar{v}$.
    \item[$B_1,B_{irr}$: ] Any choice of subset $B\subseteq N_{\bar{v}}$ with $|B|=11$.  We fix a canonical ordering on $N_{\bar{v}}$.
    \item[$B_1$:] Any choice of a crossed internal edge $\tilde{v}=(\tilde{v}_1,\tilde{v}_2)$ not adjacent to $\bar{v}$. We denote by $N_{\tilde{v}}$ the hairs and half-edges adjacent to $\tilde{v}_1$ or $\tilde{v}_2$ excluding the two crossed half-edges. $G$ is required to be simple, except possibly with loops at $\bar{v}$ or multiple edges parallel to the crossed edge.
    \item[$B_{irr}$:] $G$ is required to be simple, except possibly with loops at $\bar{v}$, and with exactly one vertex $\tilde{v}\neq\bar{v}$ with unique neighbour $\bar{v}$ and one loop, which we call the crossed edge.
    \item[\large $\substack{A_3,A_2\\ B_1,B_{irr}}$\normalsize : ] A choice of alternating order of the internal edges of $G$ which are not crossed. The resulting graph with features is required to have no odd symmetries. Here a symmetry is a graph automorphism that preserves all the choices made above.
\end{enumerate}

\begin{proposition} \label{thm:GKcombinatorial}
    For all $(g,n)$, it holds 
    \begin{multline*}
        \; \bGK^{12,1}_{g,n} = \;\;
        \bigoplus_{\left[\substack{G,\bar{v},\tilde{v} \\ |N_{\bar{v}}|\geq 3, B } \right]} \left< G_{\bar{v},B}^{\tilde{v}} \right>
        \oplus \bigg( \bigoplus_{\left[\substack{G,\bar{v},\tilde{v} \\ |N_{\tilde{v}}|=2,B }\right]} \left< G_{\bar{v},B}^{\tilde{v}} \right> \bigg) \bigg/ \{\substack{\text{weight 13}\\ \text{relations}}\} \\
        \oplus \bigg( \bigoplus_{\left[\substack{G,\bar{v} \\ B,\tilde{v}_1, \tilde{v}_2}\right]} \left<G_{\bar{v},B}^{\tilde{v}_1,\tilde{v}_2} \right> \bigg) \bigg/ \{\substack{\text{weight 11, 2, and}\\ \text{3b) relations}}\}
        \oplus \bigg( \bigoplus_{\left[G,\bar{v},\tilde{v},B\right]} \left< G_{\bar{v},B}^{\tilde{v},irr} \right> \bigg) \bigg/ \{\substack{\text{weight 11}\\ \text{relations}}\},
    \end{multline*}
    where the variables in the square brackets range over all the admissible choices described above for the four families $A_3,A_2,B_1,B_{irr}$, up to isomorphisms of graphs with these features.
\end{proposition}
\begin{proof}
    $\myB_{g,n}$ is realized as the quotient of \ref{equ:G} by the relations 3a),3b) and 4) in Section \ref{sec:SimplifiedGK}. Coinvariants relations are resolved by ignoring generators with odd symmetries.
    Relations 3a) and 4) are resolved by ignoring case $(A)$ generators with a loop at $\tilde{v}$, i.e. by requiring that all graphs in families $A_3$ and $A_2$ be simple (except possibly at $\bar{v}$). The remaining 3b) relations restrict to the family $B_1$.
\end{proof}

\subsection{Generating blown-up components} \label{subsec:genBlownup}

The algorithm starts by generating all possible connected components of the blow-up at $\bar{v}$ of decorated graphs $G_{\bar{v},B}^{\tilde{v}}$, $G_{\bar{v},B}^{\tilde{v}_1,\tilde{v}_2}$, $G_{\bar{v},B}^{\tilde{v},irr}$.

The crucial observation is that, if $C_1,..,C_k$ are the blown-up components of $G$, graph automorphisms $\phi:G\rightarrow G$ fixing $\bar{v}$ precisely correlate to permutations $\sigma\in\ss_k$ and graph isomorphisms $\phi_i:C_i\rightarrow C_{\sigma(i)}$. Thus, the isomorphism class (preserving all additional data) of a generator $G$ is uniquely determined by the list of blown-up components up to reordering and isomorphism of the components preserving the features $\epsilon,\omega$, $j$ and the crossed edge.

In addition, a generator has an odd symmetry if and only if it has either a blown-up component with an odd symmetry (when restricted to internal edges and $\epsilon$ hairs, as the sign on $\omega$ cancels out by the action on the cohomology class) or at least two isomorphic blown-up components wich have an odd number of internal edges plus $\epsilon$ hairs (again, the action on $\omega$ hairs cancels out).


The $(g,n)$ type of a decorated graph $G$ can be be computed from the $(g,n)$ type and the number of $\epsilon$, $\omega$ and $j$ hairs of each blown-up component. This means that we can first classify a set of blown-up components based on these parameters and afterwards all virtual blown-up representations with these components. \\


We describe the list of blown-up components generated in our script. It contains all the connected graphs with at least trivalent vertices, hairs labeled by $\epsilon,\omega$ or $j$ such that $0\leq 3(g-1)+3|\epsilon|+|\omega|+2|j| \leq E_{max}-25$, $|\epsilon|+|\omega|\geq 1$  and falling under exactly one of the following cases:
\begin{itemize}
    \item simple,
    \item simple and with a crossed $\omega$ hair,
    \item simple and with a crossed internal edge, possibly having maximum one multiple edge parallel to the crossed edge.
\end{itemize}

We also add manually the following seven graphs:
\[
    \begin{tikzpicture}
        \node (v) at (0,0) {$\omega$};
        \node (w) at (1,0) {$\omega$};
        \draw (v) edge[bend left=60, distance=.5cm] (w); 
    \end{tikzpicture}
    \hspace{0.4cm}
    \begin{tikzpicture}
        \node (v) at (0,0) {$\omega$};
        \node (w) at (1,0) {$\epsilon$};
        \draw (v) edge[bend left=60, distance=.5cm] (w); 
    \end{tikzpicture}
    \hspace{0.4cm}
    \begin{tikzpicture}
        \node (v) at (0,0) {$\epsilon$};
        \node (w) at (1,0) {$\epsilon$};
        \draw (v) edge[bend left=60, distance=.5cm] (w); 
    \end{tikzpicture}
    \hspace{0.4cm}
    \begin{tikzpicture}
        \node (w1) at (0,0) {$\omega$};
        \node (n) at (0,1.2) {$j$};
        \draw (w1) edge (n); 
    \end{tikzpicture}
    \hspace{0.4cm}
    \begin{tikzpicture}
        \node (w1) at (0,0) {$\epsilon$};
        \node (n) at (0,1.2) {$j$};
        \draw (w1) edge (n); 
    \end{tikzpicture}
    \hspace{0.4cm}
    \begin{tikzpicture}
        \node (w) at (0,0) {$\omega$};
        \node[int] (v) at (0,.8) {};
        \draw (w) edge (v);
        \draw (v) edge +(0,-.3) edge[loop, crossed, distance=0.8cm] (v);
    \end{tikzpicture}
    \hspace{0.4cm}
    \begin{tikzpicture}
        \node (w) at (0,0) {$\epsilon$};
        \node[int] (v) at (0,.8) {};
        \draw (w) edge (v) (v) edge[loop, crossed, distance=0.8cm] (v);
    \end{tikzpicture}
    \hspace{0.4cm}
    \begin{tikzpicture}
        \node (w) at (0,0) {$\omega$};
        \node[int] (v) at (0,.8) {};
        \draw (w) edge[crossed] (v);
        \draw (v) edge +(0,-.3) edge[loop, distance=0.8cm] (v);
    \end{tikzpicture}
\]

In order to make the task independent from the parameter $n$, we have to use a single anonymous label $j$ for all the hairs of the original graph. Recall that every way of labeling the $j$ hairs with numbers $1,...,n$ represents a distinct generator, but not in a unique way due to the presence of symmetries (which are required to fix the $j$ hairs).

In this list also lie many graphs that will be later killed by relations, for example the lone hair with two $\omega$ labels, $A_2$ graphs with a loop at $\tilde{v}$ or a a second $\omega$ hair incident to the crossed hair. Nonetheless, it is useful to keep also the redundant graphs in the list for completeness, and for later resolving the weight 2 and 3b) relations. Afterwards, they will be removed.

The blown-up components are wrapped in Python objects which compute at construction all graph parameters such as genus, valence at the special vertices, odd symmetries and Specht module contributions.


\subsection{Resolving weight 2 and 3b) relations} Graphs with a crossed internal edge (that will later generate the $B_1$ family) are grouped by the isomorphism class of their contraction at the crossed edge.
This yields groups of blown-up components such that, any weight 2 cohomology relations or the 3b) relation restrict between graphs in the same so called relation group.

It is important to keep in mind that not all decorated graphs with $\delta\{\substack{A\\A'}\}$ classes are present in our list. We have omitted the ones with loops or multiple edges incident to just one end of the crossed edge. Because they are killed by the 3b) relation or odd symmetries, they might also influence the relation group without appearing in our list.

There are three easy cases to handle automatically. If a weight 2 relation group has size one with a graph with odd symmetry, then it vanishes. If a group has size one and the contraction at the crossed edge doesn't have loops nor multiple edges, then it vanishes if and only if the single graph in the group has an odd symmetry. If a group is of minimal valence $n_{\tilde{v}}=4$, then it vanishes if and only if any of its graphs have odd symmetries or its contraction at the crossed edge has loops or multiple edges. For higher group sizes or higher valence, we manually check each weight 2 relation group and hardcode a basis into the script.

\begin{remark} \label{rmk:weight2vanishing}
    A blown-up component with crossed internal edge, a multiple edge and valence $4$ is equivalent to a blown-up component with a loop, and thus vanishes because of relation 3b). More generally, we note that if there exist two hairs or half-edges $s,t\in N_{\tilde{v}}$ in a graph $G_{\bar{v},B}^{\tilde{v}_1,\tilde{v}_2}$ with the property that it vanishes whenever the $\tilde{v}_1$,$\tilde{v}_2$ don't separate $s$ and $t$, then, for every other $x,y\in N_{\tilde{v}}$ distinct from $s,t$, the following relation holds:
\end{remark}
\[  
    \sum_{\substack{A\sqcup A'= N_{\tilde{v}}\\ s,x\in A, t,y\in A'}}
    \begin{tikzpicture}
        \node[int] (v) at (0,0) {};
        \node[int] (w) at (0.5,0) {};
        \node (s) at (-.1,.5)  {\tiny$s$};
        \node (t) at (.6,.5)  {\tiny$t$};
        \node (x) at (-.3,-.5) {\tiny$x$};
        \node (y) at (.8,-.5) {\tiny$y$};
        \draw (v) edge +(-.5,-.5) edge[dashed] +(-.5,.5) edge[dashed] +(-.5,.2) edge[dashed] +(-.5,-.2) edge[dashed] +(-.2,-.5) edge[crossed] (w) edge +(.1,.6)
        (w) edge +(.5,-.5) edge[dashed] +(.5,.5) edge[dashed] +(.5,.2) edge[dashed] +(.5,-.2) edge[dashed] +(.2,-.5) edge +(-.1,.6);
        \draw [decorate,decoration={brace,amplitude=5pt,mirror}]
        (-.6,.5) -- (-.6,-.5) node[midway,xshift=-1em]{$A$};
        \draw [decorate,decoration={brace,amplitude=5pt}]
        (1,.5) -- (1,-.5) node[midway,xshift=1em]{$A'$};
    \end{tikzpicture}
    =
    \sum_{\substack{A\sqcup A'= N_{\tilde{v}}\\ |A|,|A'|\geq 2 \\ x\in A, t,y\in A'}}
    \begin{tikzpicture}
        \node[int] (v) at (0,0) {};
        \node[int] (w) at (0.5,0) {};
        \node (t) at (.6,.5)  {\tiny$t$};
        \node (x) at (-.3,-.5) {\tiny$x$};
        \node (y) at (.8,-.5) {\tiny$y$};
        \draw (v) edge +(-.5,-.5) edge[dashed] +(-.5,.5) edge[dashed] +(-.5,.2) edge[dashed] +(-.5,-.2) edge[dashed] +(-.2,-.5) edge[crossed] (w) edge +(.1,.6)
        (w) edge +(.5,-.5) edge[dashed] +(.5,.5) edge[dashed] +(.5,.2) edge[dashed] +(.5,-.2) edge[dashed] +(.2,-.5) edge +(-.1,.6);
        \draw [decorate,decoration={brace,amplitude=5pt,mirror}]
        (-.6,.5) -- (-.6,-.5) node[midway,xshift=-1em]{$A$};
        \draw [decorate,decoration={brace,amplitude=5pt}]
        (1,.5) -- (1,-.5) node[midway,xshift=1em]{$A'$};
    \end{tikzpicture}
    =
    \sum_{\substack{A\sqcup A'= N_{\tilde{v}}\\ |A|,|A'|\geq 2 \\ x\in A, t,s\in A'}}
    \begin{tikzpicture}
        \node[int] (v) at (0,0) {};
        \node[int] (w) at (0.5,0) {};
        \node (s) at (.6,.5)  {\tiny$s$};
        \node (t) at (.4,.5)  {\tiny$t$};
        \node (x) at (-.3,-.5) {\tiny$x$};
        \draw (v) edge +(-.5,-.5) edge +(-.5,.5) edge[dashed] +(-.5,.2) edge[dashed] +(-.5,-.2) edge[dashed] +(-.2,-.5) edge[crossed] (w)
        (w) edge[dashed] +(.5,-.5) edge[dashed] +(.5,.5) edge[dashed] +(.5,.2) edge[dashed] +(.5,-.2) edge[dashed] +(.2,-.5) edge +(-.3,.6) edge +(.3,.6);
        \draw [decorate,decoration={brace,amplitude=5pt,mirror}]
        (-.6,.5) -- (-.6,-.5) node[midway,xshift=-1em]{$A$};
        \draw [decorate,decoration={brace,amplitude=5pt}]
        (1,.5) -- (1,-.5) node[midway,xshift=1em]{$A'$};
    \end{tikzpicture}
    =0  
\]
This observation can be employed to kill all graphs with crossed internal edge in excess 3 and 4 which have multiple edges and valence $n_{\tilde{v}}\geq 5$, except one. These are shown in \eqref{equ:weight2killings}, and they come in three different weight 2 relation groups. They all vanish, except for the last ones, which merely have a relation between them.

\begin{multline} \label{equ:weight2killings}
    \begin{tikzpicture}
        \node[int] (v1) at (0,0) {};
        \node[int] (v2) at (1,0) {};
        \node (w1) at (-.25,-.9) {$\omega$};
        \node (w2) at (.25,-.9) {$\omega$};
        \node (w3) at (1,-.9) {$\omega$};
        \draw (v1) edge[crossed] (v2) edge[bend left=40, distance=0.6cm] (v2) edge (w1) edge (w2) (v2) edge (w3);
    \end{tikzpicture}
    =0  \hspace{2cm}
    \begin{tikzpicture}
        \node[int] (v1) at (0,0) {};
        \node[int] (v2) at (1,0) {};
        \node[int] (v3) at (1.4,0) {};
        \node (w1) at (0,-.9) {$\omega$};
        \node (w2) at (.8,-.9) {$\omega$};
        \node (w3) at (1.2,-.9) {$\omega$};
        \node (w4) at (1.8,-.9) {$\omega$};
        \draw (v1) edge[crossed] (v2) edge[bend left=40, distance=0.6cm] (v2) edge (w1);
        \draw (v2) edge (w2) edge (v3);
        \draw (v3) edge (w3) edge (w4);
    \end{tikzpicture}
    =0    \hspace{.8cm}
    \begin{tikzpicture}
        \node[int] (v1) at (0,0) {};
        \node[int] (v2) at (1,0) {};
        \node[int] (v3) at (1.4,0) {};
        \node (w1) at (0,-.9) {$\omega$};
        \node (w2) at (.4,-.9) {$\omega$};
        \node (w3) at (1.2,-.9) {$\omega$};
        \node (w4) at (1.8,-.9) {$\omega$};
        \draw (v1) edge[crossed] (v2) edge[bend left=40, distance=0.6cm] (v2) edge (w1) edge (w2);
        \draw (v2) edge (v3);
        \draw (v3) edge (w3) edge (w4);
    \end{tikzpicture}
    =0  \\
    \begin{tikzpicture}
        \node[int] (v1) at (0,0) {};
        \node[int] (v2) at (1,0) {};
        \node (w1) at (-.25,-.9) {$\omega$};
        \node (w2) at (.25,-.9) {$\omega$};
        \node (n) at (1.4,1) {$j$};
        \draw (v1) edge[crossed] (v2)  edge[bend left=40, distance=0.6cm] (v2) edge (w1) edge (w2);
        \draw (v2) edge[bend right=30](n);
    \end{tikzpicture}
    =0 \hspace{.8cm}
    \begin{tikzpicture}
        \node[int] (v1) at (0,0) {};
        \node[int] (v2) at (1,0) {};
        \node (w1) at (0,-.9) {$\omega$};
        \node (w2) at (1,-.9) {$\omega$};
        \node (n) at (1.4,1) {$j$};
        \draw (v1) edge[crossed] (v2) edge[bend left=40, distance=0.6cm] (v2) edge (w1);
        \draw (v2) edge[bend right=30](n) edge (w2);
    \end{tikzpicture}
    =0 \hspace{2cm}
    \begin{tikzpicture}
        \node[int] (v1) at (0,0) {};
        \node[int] (v2) at (1,0) {};
        \node (w1) at (-.4,-.9) {$\omega$};
        \node (w2) at (.4,-.9) {$\omega$};
        \node (w3) at (1,-.9) {$\omega$};
        \node (w4) at (0,-.9) {$\omega$};
        \draw (v1) edge[crossed] (v2) edge[bend left=40, distance=0.6cm] (v2) edge (w1) edge (w2) edge (w4);
        \draw (v2) edge (w3);
    \end{tikzpicture}
    +
    \begin{tikzpicture}
        \node[int] (v1) at (0,0) {};
        \node[int] (v2) at (1,0) {};
        \node (w1) at (-.25,-.9) {$\omega$};
        \node (w2) at (.25,-.9) {$\omega$};
        \node (w3) at (.75,-.9) {$\omega$};
        \node (w4) at (1.25,-.9) {$\omega$};
        \draw (v1) edge[crossed] (v2) edge[bend left=40, distance=0.6cm] (v2) edge (w1) edge (w2);
        \draw (v2) edge (w3) edge (w4);
    \end{tikzpicture}
    =0
\end{multline}

Remark \ref{rmk:weight2vanishing} can also be used to kill other graphs, such as ones where the crossed edge is part of a triangle or where two half-edges being in at the same endpoint would create an odd symmetry:

\begin{equation}
    \begin{tikzpicture}
        \node[int] (v1) at (0,0) {};
        \node[int] (v2) at (1,0) {};
        \node[int] (v3) at (.5,.5) {};
        \node (w1) at (-.25,-.9) {$\omega$};
        \node (w2) at (.25,-.9) {$\omega$};
        \node (w3) at (.9,-.7) {$\omega$};
        \node (w4) at (1.25,-.9) {$\omega$};
        \draw (v1) edge[crossed] (v2) edge (v3) edge (w1) edge (w2);
        \draw (v2) edge (v3) edge (w4);
        \draw (v3) edge (w3);
    \end{tikzpicture}
    =0  \hspace{1cm}
    \begin{tikzpicture}
        \node[int] (v1) at (0,0) {};
        \node[int] (v2) at (1,0) {};
        \node[int] (v3) at (1.75,0) {};
        \node[int] (v4) at (2.5,0) {};
        \node (w1) at (-.25,-.9) {$\omega$};
        \node (w2) at (.25,-.9) {$\omega$};
        \node (w3) at (.75,-.9) {$\omega$};
        \node (w4) at (1.25,-.9) {$\omega$};
        \node (w5) at (1.75,-.9) {$\omega$};
        \node (w6) at (2.25,-.9) {$\omega$};
        \node (w7) at (2.75,-.9) {$\omega$};
        \draw (v1) edge(v2)  edge (w1) edge (w2) (v2) edge[crossed] (v3)  edge (w3) edge (w4);
        \draw (v3) edge (v4)  edge (w5) (v4) edge(w6)edge(w7);
    \end{tikzpicture}
    =0.
\end{equation}

\eqref{equ:weight2groupof4} is another example in excess 4 of a weight 2 relation group. It contains the four graphs which appear as terms in the equation, and one checks that all the relations turn out equivalent to the one displayed. Thus, a basis for this group is determined by any three graphs.
\begin{equation} \label{equ:weight2groupof4}
    \begin{tikzpicture}
        \node[int] (v1) at (0,0) {};
        \node[int] (v2) at (.75,0) {};
        \node[int] (v3) at (1.5,0) {};
        \node (j) at (.75,.9) {\small$j$};
        \node (w1) at (-.4,-.8) {\small$\omega$};
        \node (w2) at (0,-.8) {\small$\omega$};
        \node (w3) at (.4,-.8) {\small$\omega$};
        \node (w4) at (1.3,-.8) {\small$\omega$};
        \node (w5) at (1.7,-.8) {\small$\omega$};
        \draw (v1) edge[crossed](v2) edge(w1)edge(w2)edge(w3);
        \draw (v2) edge(v3) edge(j)  (v3) edge(w4)edge(w5);
    \end{tikzpicture} +
    \begin{tikzpicture}
        \node[int] (v1) at (0,0) {};
        \node[int] (v2) at (.75,0) {};
        \node[int] (v3) at (1.5,0) {};
        \node (j) at (.75,.9) {\small$j$};
        \node (w1) at (-.2,-.8) {\small$\omega$};
        \node (w2) at (.2,-.8) {\small$\omega$};
        \node (w3) at (.75,-.8) {\small$\omega$};
        \node (w4) at (1.3,-.8) {\small$\omega$};
        \node (w5) at (1.7,-.8) {\small$\omega$};
        \draw (v1) edge[crossed](v2) edge(w1)edge(w2);
        \draw (v2) edge(v3) edge(w3) edge(j)  (v3) edge(w4)edge(w5);
    \end{tikzpicture} -
    \begin{tikzpicture}
        \node[int] (v1) at (0,0) {};
        \node[int] (v2) at (.75,0) {};
        \node[int] (v3) at (1.5,0) {};
        \node (j) at (1.5,.9) {\small$j$};
        \node (w1) at (-.2,-.8) {\small$\omega$};
        \node (w2) at (.2,-.8) {\small$\omega$};
        \node (w3) at (.75,-.8) {\small$\omega$};
        \node (w4) at (1.3,-.8) {\small$\omega$};
        \node (w5) at (1.7,-.8) {\small$\omega$};
        \draw (v1) edge(v2) edge(w1)edge(w2);
        \draw (v2) edge[crossed](v3) edge(w3)  (v3) edge(w4)edge(w5) edge(j);
    \end{tikzpicture} -
    \begin{tikzpicture}
        \node[int] (v1) at (0,0) {};
        \node[int] (v2) at (.75,0) {};
        \node[int] (v3) at (1.5,0) {};
        \node (j) at (1.5,.9) {\small$j$};
        \node (w1) at (-.4,-.8) {\small$\omega$};
        \node (w2) at (.1,-.8) {\small$\omega$};
        \node (w3) at (.5,-.8) {\small$\omega$};
        \node (w4) at (1,-.8) {\small$\omega$};
        \node (w5) at (1.5,-.8) {\small$\omega$};
        \draw (v1) edge(v2) edge(w1)edge(w2);
        \draw (v2) edge[crossed](v3) edge(w3)edge(w4)  (v3) edge(w5) edge(j);
    \end{tikzpicture}
    =0
\end{equation}

\subsection{Generating virtual blown-up representations}  First we remove odd symmetries, weight 2 and 3b) relations, and $A_2$ graphs which are related to $B_{irr}$ graphs through 3a) and 4) relations, obtaining the list called \emph{blown-up-components-basis} in our script. Then, we build so called \textit{virtual blown-up graphs}. These are all unordered lists $\{C_1,...,C_k\}$ of graphs in blown-up-components-basis with:
\begin{itemize}
    \item exactly one component $C_i$ is crossed,
    \item $\omega_i\leq 11$, $e(C_i)=3(g_i-1)+3\epsilon_i+\omega_i+2n_i \geq 1$, $\sum_i e(C_i) \leq E_{max}-25$,
    \item no duplicates of components that create an odd symmetry when exchanged by an automorphism.
\end{itemize}
    
We explain the last condition. Two isomorphic blown-up components (which don't individually have an odd symmetry already) create an odd symmetry if and only if they have no $j$ hairs (because automorphisms don't exchange the $j$ hairs) and have an odd number of internal edges plus $\epsilon$ hairs. Also the lone hair with two $\epsilon$ labels is of this type.

\subsection{Weight 11 relations} The first step of the program is actually to generate so called \textit{unmarked blown-up components}, which have a unique label $u$ for the half-edges that where incident to $\bar{v}$. Only afterwards, blown-up components with $\omega$ and $\epsilon$ hairs are generated by marking the $u$ hairs in every possible way. The unmarked component which was used as a template for the marking is stored in each respective blown-up component. In so doing, a list of blown-up components can be obtained from another by permuting its $\omega$ and $\epsilon$ hairs if and only if they have the same list of unmarked components and same number of total $\epsilon$ hairs, which can be instantly checked instead of testing for isomorphism.

Thus, we group the virtual blown-up representations (of families $B_1$,$B_{irr}$) by their list of unmarked components and number of $\epsilon$ hairs. Every weight 11 relation ren restricts between graphs within the same so called relation group.
In addition, this can be done at the virtual level without pinning down a $(g,n)$ pair in the excess class by ignoring in the grouping process the components obtained as a marking of a lone $j$ hair or a tripod. \\

For example, the unmarked component in \eqref{equ:weight11relationexample} determines in excess 3 (the excess determines the number of $\epsilon$ hairs) a weight 11 relation group of three virtual blown-up graphs. Note that not all graphs in the same relation group must have the same existence range: the first and third exist for $(4,8),(6,5),(8,2)$, whereas the middle one only for $(6,5),(8,2)$.
The corresponding weight 11 relation is an integer-weighted alternating sum of the three or two graphs, depending upon the $(g,n)$ pair. Thus, discarding the third graph leaves us with a basis of the relation group for all three $(g,n)$ pairs.
\begin{equation} \label{equ:weight11relationexample}
    \begin{tikzpicture}
        \node (w1) at (0,-.5) {$u$};
        \node (w2) at (.5,-.5) {$u$};
        \node (w3) at (1,-.5) {$u$};
        \node (w4) at (1.5,-.5) {$u$};
        \node[int] (v1) at (.25,.5) {};
        \node[int] (v2) at (1.25,.5) {};
        \draw(v1) edge[crossed](v2)edge(w1)edge(w2) (v2) edge(w3)edge(w4);
    \end{tikzpicture}
    \xlongrightarrow{\substack{\text{marked components}\\\text{with one $\epsilon$ hair}}} \hspace{.8cm}
    \begin{tikzpicture}
        \node (w1) at (0,-.5) {$\omega$};
        \node (w2) at (.5,-.5) {$\omega$};
        \node (w3) at (1,-.5) {$\omega$};
        \node (w4) at (1.5,-.5) {$\omega$};
        \node[int] (v1) at (.25,.5) {};
        \node[int] (v2) at (1.25,.5) {};
        \draw(v1) edge[crossed](v2)edge(w1)edge(w2) (v2) edge(w3)edge(w4);
        \node (e) at (-.6,-.5) {$\epsilon$};
        \node (j) at (-.6,.7) {$j$};
        \draw(e) edge(j);
    \end{tikzpicture},
    \begin{tikzpicture}
        \node (w1) at (0,-.5) {$\omega$};
        \node (w2) at (.5,-.5) {$\omega$};
        \node (w3) at (1,-.5) {$\omega$};
        \node (w4) at (1.5,-.5) {$\omega$};
        \node[int] (v1) at (.25,.5) {};
        \node[int] (v2) at (1.25,.5) {};
        \draw(v1) edge[crossed](v2)edge(w1)edge(w2) (v2) edge(w3)edge(w4);
        \node (ww1) at (-1.5,-.5) {$\omega$};
        \node (ww2) at (-1,-.5) {$\omega$};
        \node (ww3) at (-.5,-.5) {$\epsilon$};
        \node[int] (vv) at (-1,.5) {};
        \draw(vv) edge(ww1)edge(ww2)edge(ww3);
    \end{tikzpicture},
    \begin{tikzpicture}
        \node (w1) at (0,-.5) {$\omega$};
        \node (w2) at (.5,-.5) {$\omega$};
        \node (w3) at (1,-.5) {$\omega$};
        \node (w4) at (1.5,-.5) {$\epsilon$};
        \node[int] (v1) at (.25,.5) {};
        \node[int] (v2) at (1.25,.5) {};
        \draw(v1) edge[crossed](v2)edge(w1)edge(w2) (v2) edge(w3)edge(w4);
    \end{tikzpicture}
\end{equation}

\subsection{Weight 13 relations} The process is similar to weight 11, only that now we have to determine virtual blown-up graphs up to permutation of the $\omega$ hairs with the two hairs or half-edges at $\tilde{v}$. This is done again by blowing up at $\tilde{v}$:  the crossed edge and $\tilde{v}$ are discarded, the two newly created hairs are marked $\omega$, and we end up with a list of uncrossed blown-up components with a total of 12 $\omega$ hairs.

Thus, we group by the list of so called \textit{$A_2$ components} to determine weight 13 relation groups of $A_2$ virtual blown-up representations. Again, this can be done at the virtual level by ignoring the lone $\omega-j$ hairs and the tripods with $\omega$ labels. \\

For example, the two blown-up components in \eqref{equ:weight13relationexample} determine in excess 4 a weight 13 relation group of seven virtual blown-up representations. One can check that they all have the same list of blown-up components after removing $\tilde{v}$ (up to lone hairs and tripods).  The graphs with two tripods glued to $\tilde{v}$ (or with two hairs of the same tripod) are not listed because they would have an odd symmetry. Note again that not all virtual blown-up representations have the same existence range.

\begin{multline} \label{equ:weight13relationexample}
    \begin{tikzpicture}
        \node (w1) at (0,-.5) {\small$\omega$};
        \node (w2) at (.4,-.5) {\small$\omega$};
        \node (w3) at (.8,-.5) {\small$\omega$};
        \node (w4) at (1.2,-.5) {\small$\omega$};
        \node[int] (v1) at (.2,.5) {};
        \node[int] (v2) at (1,.5) {};
        \draw(v1) edge(v2)edge(w1)edge(w2) (v2) edge(w3)edge(w4);
        \node (w5) at (1.6,-.5) {\small$\omega$};
        \node (w6) at (2,-.5) {\small$\omega$};
        \node[int] (v5) at (1.8,.3) {};
        \node (j) at (1.8,1.2) {\small $j$};
        \draw (v5) edge(j)edge(w5)edge(w6);
    \end{tikzpicture}
    \xlongrightarrow{\text{$A_2$ blown-ups}}
    \begin{tikzpicture}
        \node (w1) at (0,-.5) {\small$\omega$};
        \node (w2) at (.4,-.5) {\small$\omega$};
        \node (w3) at (.8,-.5) {\small$\omega$};
        \node[int] (v1) at (.2,.5) {};
        \node[int] (v2) at (1,.5) {};
        \draw(v1) edge(v2)edge(w1)edge(w2) (v2) edge(w3);
        \node (w6) at (2.4,-.5) {\small$\omega$};
        \node[int] (v5) at (2.2,.3) {};
        \node (j) at (2.2,1.2) {\small $j$};
        \draw (v5) edge(j)edge(w6);
        \node (ww) at (1.6,-.5) {\small$\omega$};
        \node[int] (CR) at (1.6,.3) {};
        \draw (CR) edge[crossed](ww) edge(v2)edge(v5);
    \end{tikzpicture},
    \begin{tikzpicture}
        \node (w1) at (0,-.5) {\small$\omega$};
        \node (w2) at (.4,-.5) {\small$\omega$};
        \node (w3) at (.8,-.5) {\small$\omega$};
        \node[int] (v1) at (.2,.5) {};
        \node[int] (v2) at (1,.5) {};
        \draw(v1) edge(v2)edge(w1)edge(w2) (v2) edge(w3);
        \node (w5) at (2,-.5) {\small$\omega$};
        \node (w6) at (2.4,-.5) {\small$\omega$};
        \node[int] (v5) at (2.2,.3) {};
        \node (j) at (2.2,1.2) {\small $j$};
        \draw (v5) edge(j)edge(w5)edge(w6);
        \node (ww) at (1.6,-.5) {\small$\omega$};
        \node[int] (CR) at (1.6,.3) {};
        \node (j1) at (1.6,1.2) {\small $j$};
        \draw (CR) edge[crossed](ww) edge(v2)edge(j1);
    \end{tikzpicture},
    \begin{tikzpicture}
        \node (w1) at (0,-.5) {\small$\omega$};
        \node (w2) at (.4,-.5) {\small$\omega$};
        \node (w3) at (.8,-.5) {\small$\omega$};
        \node[int] (v1) at (.2,.5) {};
        \node[int] (v2) at (1,.5) {};
        \draw(v1) edge(v2)edge(w1)edge(w2) (v2) edge(w3);
        \node (w5) at (2.8,-.5) {\small$\omega$};
        \node (w6) at (3.2,-.5) {\small$\omega$};
        \node[int] (v5) at (3,.3) {};
        \node (j) at (3,1.2) {\small $j$};
        \draw (v5) edge(j)edge(w5)edge(w6);
        \node (ww) at (1.6,-.5) {\small$\omega$};
        \node[int] (CR) at (1.6,.3) {};
        \node (ww1) at (2,-.5) {$\omega$};
        \node (ww2) at (2.4,-.5) {$\omega$};
        \node[int] (vv) at (2,.3) {};
        \draw(vv) edge(ww1)edge(ww2);
        \draw (CR) edge[crossed](ww) edge(v2)edge(vv);
    \end{tikzpicture}, \\
    \begin{tikzpicture}
        \node (w1) at (0,-.5) {\small$\omega$};
        \node (w2) at (.4,-.5) {\small$\omega$};
        \node (w3) at (.8,-.5) {\small$\omega$};
        \node (w4) at (1.2,-.5) {\small$\omega$};
        \node[int] (v1) at (.2,.5) {};
        \node[int] (v2) at (1,.5) {};
        \draw(v1) edge(v2)edge(w1)edge(w2) (v2) edge(w3)edge(w4);
        \node (w6) at (2.4,-.5) {\small$\omega$};
        \node[int] (v5) at (2.2,.3) {};
        \node (j) at (2.2,1.2) {\small $j$};
        \draw (v5) edge(j)edge(w6);
        \node (ww) at (1.6,-.5) {\small$\omega$};
        \node[int] (CR) at (1.6,.3) {};
        \node (j1) at (1.6,1.2) {\small $j$};
        \draw (CR) edge[crossed](ww) edge(j1)edge(v5);
    \end{tikzpicture},
    \begin{tikzpicture}
        \node (w1) at (0,-.5) {\small$\omega$};
        \node (w2) at (.4,-.5) {\small$\omega$};
        \node (w3) at (.8,-.5) {\small$\omega$};
        \node (w4) at (1.2,-.5) {\small$\omega$};
        \node[int] (v1) at (.2,.5) {};
        \node[int] (v2) at (1,.5) {};
        \draw(v1) edge(v2)edge(w1)edge(w2) (v2) edge(w3)edge(w4);
        \node (w6) at (3.2,-.5) {\small$\omega$};
        \node[int] (v5) at (3,.3) {};
        \node (j) at (3.2,1.2) {\small $j$};
        \draw (v5) edge(j)edge(w6);
        \node (ww1) at (1.6,-.5) {$\omega$};
        \node (ww2) at (2,-.5) {$\omega$};
        \node[int] (vv) at (2,.3) {};
        \draw(vv) edge(ww1)edge(ww2);
        \node (ww) at (2.4,-.5) {\small$\omega$};
        \node[int] (CR) at (2.4,.3) {};
        \draw (CR) edge[crossed](ww) edge(vv)edge(v5);
    \end{tikzpicture},
    \begin{tikzpicture}
        \node (w1) at (0,-.5) {\small$\omega$};
        \node (w2) at (.4,-.5) {\small$\omega$};
        \node (w3) at (.8,-.5) {\small$\omega$};
        \node (w4) at (1.2,-.5) {\small$\omega$};
        \node[int] (v1) at (.2,.5) {};
        \node[int] (v2) at (1,.5) {};
        \draw(v1) edge(v2)edge(w1)edge(w2) (v2) edge(w3)edge(w4);
        \node (w5) at (2.8,-.5) {\small$\omega$};
        \node (w6) at (3.2,-.5) {\small$\omega$};
        \node[int] (v5) at (3,.3) {};
        \node (j) at (3,1.2) {\small $j$};
        \draw (v5) edge(j)edge(w5)edge(w6);
        \node (ww1) at (1.6,-.5) {$\omega$};
        \node (ww2) at (2,-.5) {$\omega$};
        \node[int] (vv) at (2,.3) {};
        \draw(vv) edge(ww1)edge(ww2);
        \node (jj) at (2.4,1.2) {\small$j$};
        \node (ww) at (2.4,-.5) {\small$\omega$};
        \node[int] (CR) at (2.4,.3) {};
        \draw (CR) edge[crossed](ww) edge(vv) edge(jj);
    \end{tikzpicture},
    \begin{tikzpicture}
        \node (w1) at (0,-.5) {\small$\omega$};
        \node (w2) at (.4,-.5) {\small$\omega$};
        \node (w3) at (.8,-.5) {\small$\omega$};
        \node (w4) at (1.2,-.5) {\small$\omega$};
        \node[int] (v1) at (.2,.5) {};
        \node[int] (v2) at (1,.5) {};
        \draw(v1) edge(v2)edge(w1)edge(w2) (v2) edge(w3)edge(w4);
        \node (w5) at (2,-.5) {\small$\omega$};
        \node (w6) at (2.4,-.5) {\small$\omega$};
        \node[int] (v5) at (2.2,.3) {};
        \node (j) at (2.2,1.2) {\small $j$};
        \draw (v5) edge(j)edge(w5)edge(w6);
        \node (jj) at (1.6,1.2) {\small$j$};
        \node (jjj) at (1.2,1.2) {\small$j$};
        \node (ww) at (1.6,-.5) {\small$\omega$};
        \node[int] (CR) at (1.6,.3) {};
        \draw (CR) edge[crossed](ww) edge(jjj) edge(jj);
    \end{tikzpicture},
\end{multline}

\subsection{Completing virtual blown-up representations} The script also offers the possibility to focus on a particular $(g,n)$ pair by \textit{completing} each virtual blown-up representation in an excess. This means focusing on the ones with $(g,n)$ in their existence range and adding the right amount of $\omega-j$ hairs and tripods to their list of blown-up components.

In the results folder are also saved lists of completed blown-up representations for every $(g,n)$ pair with $E(g,n)\leq 29$; but, again, we have not verified lists in excess $29$.

Having at disposal the Specht module contributions of each completed blown-up representation, one can compute automatically the Euler Characteristic for each $(g,n)$ pair individually, but with a caveat: weight 11 and 13 relations haven't been resolved, so one must account for that.

\bibliographystyle{amsplain}
\bibliography{refs}

\providecommand{\bysame}{\leavevmode\hbox to3em{\hrulefill}\thinspace}
\providecommand{\MR}{\relax\ifhmode\unskip\space\fi MR }
\providecommand{\MRhref}[2]{%
  \href{http://www.ams.org/mathscinet-getitem?mr=#1}{#2}
}
\providecommand{\href}[2]{#2}
\begin{thebibliography}{10}

\bibitem{github}
\url{https://github.com/bellimarco/Getzler-Kapranov-Graph-Cohomology-Computations-in-weight-13}


\bibitem{BergstromFaberPayne}
Jonas Bergstr\"{o}m, Carel Faber, and Sam Payne, \emph{Polynomial point counts
  and odd cohomology vanishing on moduli spaces of stable curves}, Ann. of
  Math. (2) \textbf{199} (2024), no.~3, 1323--1365. \MR{4740541}

\bibitem{CLP}
Samir Canning, Hannah Larson, and Sam Payne, \emph{The eleventh cohomology
  group of {$\overline{\mathcal{M}}_{g,n}$}}, Forum Math. Sigma \textbf{11}
  (2023), no.~Paper No. e62, 18 pp.

\bibitem{CLPW}
Samir Canning, Hannah Larson, Sam Payne, and Thomas Willwacher, \emph{Moduli
  spaces of curves with polynomial point counts}, preprint, arXiv:2410.19913,
  2024.

\bibitem{CLPW2}
Samir Canning, Hannah Larson, Sam Payne, and Thomas Willwacher, \emph{The motivic structures $\mathsf{LS}_{12}$ and $\mathsf{S}_{16}$ in the cohomology of moduli spaces of curves}, preprint, arXiv:2411.12652,
  2024.


\bibitem{PayneWillwacher21}
Sam Payne and Thomas Willwacher, \emph{The weight two compactly supported
  cohomology of moduli spaces of curves}, To appear in Duke Math. J.
  arXiv:2110.05711v1, 2021.


\end{thebibliography}

\includepdf[pages=1,fitpaper=true,pagecommand={}]{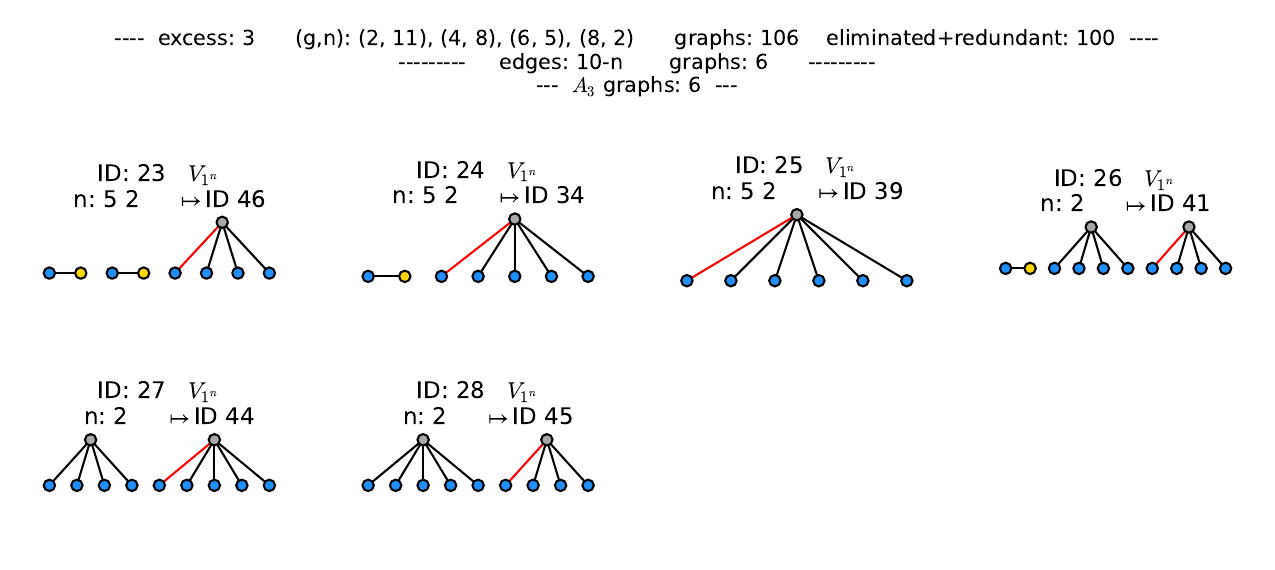}
\includepdf[pages=2,fitpaper=true,pagecommand={}]{Results/Excess3/VirtualBlowups_excess3.pdf}
\includepdf[pages=3-5,fitpaper=true,pagecommand={}]{Results/Excess3/VirtualBlowups_excess3.pdf}
\includepdf[pages=6,fitpaper=true,pagecommand={}]{Results/Excess3/VirtualBlowups_excess3.pdf}
\includepdf[pages=7-8,fitpaper=true,pagecommand={}]{Results/Excess3/VirtualBlowups_excess3.pdf}
\includepdf[pages=9,fitpaper=true,pagecommand={}]{Results/Excess3/VirtualBlowups_excess3.pdf}

\end{document}